\documentclass{amsart}

\usepackage{amssymb}
\usepackage[T1]{fontenc}

\newtheorem{theorem}{Theorem}[section]
\newtheorem{lemma}[theorem]{Lemma}
\newtheorem{proposition}[theorem]{Proposition}

\numberwithin{equation}{section}

\newcommand{\R}{\mathbf{R}}
\newcommand{\Z}{\mathbf{Z}}
\newcommand{\cG}{\mathcal{G}}
\newcommand{\cL}{\mathcal{L}}
\newcommand{\cS}{\mathcal{S}}
\newcommand{\eps}{\varepsilon}
\DeclareMathOperator{\supp}{\operatorname{supp}}
\DeclareMathOperator{\diam}{\operatorname{diam}}

\begin{document}

\title[Existence of a separating norm]
{Existence of a separating norm for the semilinear heat equation}

\author[J. Takahashi]{Jin Takahashi}
\address{Department of Mathematical and Computing Science, 
Institute of Science Tokyo, Tokyo 152-8552, Japan}
\email{takahashi.j.cf61@m.isct.ac.jp}

\subjclass[2020]{Primary 35K58, Secondary 35A01, 35B44, 35K15}
\keywords{Semilinear heat equation, global existence, blow-up, separating norm.}

\begin{abstract}
We consider the initial value problem of the semilinear heat equation 
in the Fujita supercritical range 
for which global-in-time solutions may exist. 
The main result shows the existence of 
a separating norm $\|\cdot\|_X$ in the sense that, 
for nonnegative initial data $u_0\in L^\infty(\R^n)$, 
the smallness of $\|u_0\|_X$ implies global-in-time existence,  
whereas the largeness of $\|u_0\|_X$ implies finite-time blow-up.  
The separating norm is unique up to equivalence 
on nonnegative data. 
\end{abstract}

\maketitle

\section{Introduction}
We consider the initial value problem of the following semilinear heat equation: 
\begin{equation}\label{eq:main}
\left\{ \begin{aligned}
	&u_t-\Delta u=u^p, && x\in \R^n,\; t>0, \\
	&u(x,0)=u_0(x), && x\in \R^n, \\
\end{aligned}
\right.
\end{equation}
where $n\geq1$ and  $p>1$. 
Throughout this paper, we assume that 
$u_0\in L^\infty(\R^n)$ is nonnegative 
and we only consider nonnegative solutions. 
The problem \eqref{eq:main} admits a unique classical $L^\infty$-solution 
with maximal existence time 
$0<T_*(u_0)\leq \infty$, 
see \cite[Definition 15.1, Propositions 16.1, 51.40]{QSbook2} for instance. 
By the pioneering work of Fujita \cite{Fu66} 
and subsequent developments \cite{Ha73,Su75}, 
it is well known that the exponent 
\[
	p_F:=\frac{n+2}{n}
\]
is critical for the occurrence of $T_*(u_0)=\infty$ in the sense that 
if $p\leq p_F$, then 
$T_*(u_0)< \infty$ for all $u_0\not \equiv 0$. 
If $p> p_F$, then 
$T_*(u_0)= \infty$ for some $u_0\not \equiv 0$.

In this paper, we focus on the case $p>p_F$ and 
separate global-in-time existence 
from finite-time blow-up 
in terms of the size of $u_0$ measured in a single norm. 
As in \cite[Subsection 20.5]{QSbook2}, 
such a norm is called a separating norm. 
More precisely, we say that a norm $\|\cdot\|$ 
is a separating norm if 
there exist constants $0<c<C<\infty$ satisfying the following: 
for nonnegative initial data $u_0\in L^\infty(\R^n)$, 
\[
\left\{
\begin{aligned}
	&\mbox{if $\|u_0\| \leq c$, then $T_*(u_0) =\infty$};  \\
	&\mbox{if $C\leq \|u_0\|\leq \infty$, then $T_*(u_0) <\infty$}. 
\end{aligned}
\right.
\]
By the $1$-homogeneity of norms, 
if a separating norm exists, 
then it is unique up to equivalence on nonnegative functions. 
The problem of finding a separating norm 
for nonnegative initial data $u_0\in L^\infty(\R^n)$ 
is listed in \cite[OP 2.2]{QSbook2}. 
A separating norm can be implicitly defined by 
$\| f\|:= \inf \{ a>0;\; |f|/a \in \cG\}$, 
where $\cG:= \{ f\in L^\infty(\R^n);\; f\geq0, T_*(f)=\infty\}$. 
However, our goal is to give an explicit one.

A candidate for an explicit separating norm 
is a Morrey norm 
\[
	\|f\|_{M^{q, \frac{2q}{p-1}}(\R^n)}
	:= \left( \sup_{x\in \R^n, \; r>0} r^{\frac{2q}{p-1}-n} 
	\int_{B_r(x)} |f(y)|^q dy \right)^\frac{1}{q}
	\quad \left( 1\leq q<\frac{n(p-1)}{2} \right), 
\]
where $n(p-1)/2>1$ for $p>p_F$ 
and $B_r(x):=\{ y\in \R^n; \; |x-y|<r\}$ for $x\in \R^n$ and $r>0$. 
We note that this is an invariant norm under the scaling 
\begin{equation}\label{eq:fscaling}
	f_\lambda(x) := \lambda^\frac{2}{p-1} f(\lambda x) 
	\quad (\lambda>0), 
\end{equation}
which corresponds to the natural scaling 
$u_\lambda(x,t):= \lambda^{2/(p-1)}u(\lambda x, \lambda^2 t)$ 
of \eqref{eq:main}. 
The Morrey norms almost have a separating property 
in the sense that 
\[
\left\{
\begin{aligned}
	&\mbox{if $\|u_0\|_{M^{q, \frac{2q}{p-1}}(\R^n)} \leq c$ 
	for some $q>1$, 
	then $T_*(u_0) =\infty$};  \\
	&\mbox{if $C\leq \|u_0\|_{M^{1, \frac{2}{p-1}}(\R^n)}\leq \infty$, 
	then $T_*(u_0) <\infty$}. 
\end{aligned}
\right.
\]
The global existence part is due to Souplet \cite[Proposition 6.1]{So17}  
and the blow-up part is due to Biler \cite[Corollary 3.8 (b)]{Bi18}. 
For further expositions, see \cite[Subsection 20.5]{QSbook2}. 
It remained an open question 
whether one could take $q=1$ in the global existence part. 
However, Theorem \ref{th:main} below shows that 
the choice $q=1$ is impossible, 
since the separating norm determined in this paper  
is strictly stronger than the $M^{1, 2/(p-1)}(\R^n)$-norm.

We introduce the following space and norm, 
which are given explicitly by 
\begin{align}
	&\notag
	X:=  \{ f\in L^\infty (\R^n); \, 
	\| f\|_X <\infty\}, \\
	&\label{eq:Xdef}
	\|f\|_X:= 
	\sup_{\substack{
	\phi\in C_0^\infty(\R^n)\\
	\phi\not\equiv 0 }}
	\frac{ \displaystyle 
	\int_{\R^n} |f(z)| |\phi(z)|^\frac{p}{p-1} dz}
	{ \displaystyle 
	\int_0^\infty \int_{\R^n} 
	\left| \int_{\R^n} \Delta_y G(y-z,s) \phi(z) dz \right|^\frac{p}{p-1} dy ds }, 
\end{align}
where $G(x,t):= (4\pi t)^{-n/2} e^{-|x|^2/(4t)}$ 
is the heat kernel on $\R^n$. 
By \cite[Proposition 2.30, Theorem 2.34]{BCDbook}, 
the denominator in \eqref{eq:Xdef} 
can be characterized 
by the $\dot{B}^{2/p}_{p/(p-1),p/(p-1)}$-norm of $\phi$. 
The basic properties of $X$ 
such as the scaling invariance 
and inclusion relations 
are summarized in Section \ref{sec:property}. 
In particular, Proposition \ref{pro:rel} below shows that 
$X$ lies between the scaling invariant Lorentz space 
and Morrey space. More precisely, 
\[
	L^{\frac{n(p-1)}{2},\infty}(\R^n) 
	\cap 
	L^\infty(\R^n) 
	\subsetneq X 
	\subsetneq M^{1,\frac{2}{p-1}}(\R^n) 
	\cap L^\infty(\R^n), 
\]
and so the $X$-norm is stronger than the $M^{1, 2/(p-1)}(\R^n)$-norm. 
Our main result shows that 
the $X$-norm is a separating norm.

\begin{theorem}\label{th:main}
Let $n\geq1$ and $p>p_F$. 
Then $\|\cdot\|_X$ is a 
separating norm. 
It is unique  up to equivalence 
on nonnegative functions. 
\end{theorem}

In the rest of this paper, 
we summarize properties of $X$ in Section \ref{sec:property}. 
The proof of Theorem \ref{th:main} 
is given in Sections \ref{sec:gex} and \ref{sec:ex}. 
In Section \ref{sec:gex}, 
by using a suitable choice of test function 
based on the backward heat equation, 
we show that 
the initial data of global-in-time solutions satisfy
$\| u_0\|_X \leq C_p$ 
with an explicit $C_p>0$ depending only on $p$. 
In Section \ref{sec:ex}, 
we construct a global-in-time solution 
for $\|u_0\|_X\leq c $ with $c>0$ small. 
The method is based on the construction of a supersolution 
to the integral equation corresponding to \eqref{eq:main}.

\section{Properties of the function space}\label{sec:property}
In this section, we summarize the properties of the space $(X,\|\cdot\|_X)$. 
Let $G(x,t)= (4\pi t)^{-n/2} e^{-|x|^2/(4t)}$ 
be the heat kernel on $\R^n$. 
First, we see that the denominator in \eqref{eq:Xdef} 
is positive and finite for $\phi\not \equiv 0$. 

\begin{lemma}
For any Schwartz function $\phi\in \cS(\R^n)$ with $\phi\not \equiv 0$, 
\[
	0<
	\int_0^\infty \int_{\R^n} 
	\left| \int_{\R^n} \Delta_y G(y-z,s) \phi(z) dz \right|^\frac{p}{p-1} dy ds
	<\infty. 
\]
In particular, 
the fraction in \eqref{eq:Xdef} 
is finite for $\phi\in C_0^\infty(\R^n)$ with 
$\phi\not \equiv 0$. 
\end{lemma}

\begin{proof}
By the integration by parts and the smoothing estimate, 
\[
\begin{aligned}
	&\int_0^1 \int_{\R^n} 
	\left| \int_{\R^n} \Delta_y G(y-z,s) \phi(z) dz \right|^\frac{p}{p-1} dy ds \\
	&= 
	\int_0^1 \int_{\R^n} 
	\left| \int_{\R^n}  G(y-z,s) \Delta_z \phi(z) dz \right|^\frac{p}{p-1} dy ds 
	\leq 
	\| \Delta \phi\|_{L^\frac{p}{p-1}(\R^n)}^\frac{p}{p-1}. 
\end{aligned}
\]
From the Young inequality and direct computations, it follows that 
\[
\begin{aligned}
	&\int_1^\infty \int_{\R^n} 
	\left| \int_{\R^n} \Delta_y G(y-z,s) \phi(z) dz \right|^\frac{p}{p-1} dy ds \\
	&\leq
	\int_1^\infty 
	\| \Delta G(\cdot,s) \|_{L^\frac{p}{p-1}(\R^n)}^\frac{p}{p-1} 
	\| \phi \|_{L^1(\R^n)}^\frac{p}{p-1} ds 
	\leq 
	C\| \phi \|_{L^1(\R^n)}^\frac{p}{p-1} \int_1^\infty  
	s^{-\frac{n}{2}\cdot \frac{1}{p-1} - \frac{p}{p-1}} ds. 
\end{aligned}
\]
Since the above time integral is finite, we obtain 
\begin{equation}\label{eq:forcutoff}
\begin{aligned}
	&\int_0^\infty 
	\int_{\R^n} 
	\left| \int_{\R^n} \Delta_y G(y-z,s) \phi(z) dz 
	\right|^\frac{p}{p-1} dy ds \\
	&\leq C\| \Delta \phi \|_{L^\frac{p}{p-1}(\R^n)}^\frac{p}{p-1}
	+ C\| \phi \|_{L^1(\R^n)}^\frac{p}{p-1} 
	\quad \mbox{ for }\phi \in \cS(\R^n), 
\end{aligned}
\end{equation}
where $C>0$ depends only on $n$ and $p$. 
Thus, the finiteness in the statement follows. 

To prove the positivity for $\phi\not \equiv 0$, 
it suffices to show that if 
\[
	\int_0^\infty 
	\int_{\R^n} 
	\left| \int_{\R^n} \Delta_y G(y-z,s) \phi(z) dz 
	\right|^\frac{p}{p-1} dy ds =0, 
\]
then $\phi\equiv 0$. 
Under this assumption, the integration by parts gives 
$\int_{\R^n} G(y-z,s) \Delta_z \phi(z) dz =0$ for $s>0$. 
Letting $s\to 0$ yields $\Delta \phi=0$ in $\R^n$. 
Since $\phi\in \cS(\R^n)$, we obtain $\phi\equiv 0$. 
This implies the desired positivity. 
\end{proof}

Next, we check that 
$X$ is trivial in the case $p\leq p_F$. 

\begin{lemma}\label{lem:trivial}
Let $1<p\leq p_F$. Then $X=\{0\}$. 
\end{lemma}

\begin{proof}
We first consider the case $1<p<p_F$. 
Let $f\in X$ and let $\phi\in C_0^\infty(\R^n)$ 
satisfy $0\leq \phi\leq 1$ 
and $\phi=1$ on $B_1(0)$. 
For $R>1$, set $\phi_R(x):= \phi(x/R)$. 
From \eqref{eq:Xdef} 
and the definition of the heat kernel 
$G(x,t)= (4\pi t)^{-n/2} e^{-|x|^2/(4t)}$, 
and then from $n-(2/(p-1)) <0$ by $p<p_F$, it follows that 
\[
\begin{aligned}
	&\int_{\R^n} |f(z)| |\phi_R(z)|^\frac{p}{p-1} dz 
	\leq \| f\|_X 
	\int_0^\infty \int_{\R^n} 
	\left| \int_{\R^n} \Delta_y G(y-z,s) \phi_R(z) dz \right|^\frac{p}{p-1} dy ds \\
	&= 
	\| f\|_X 
	\int_0^\infty \int_{\R^n} 
	\left| R^{-2} \int_{\R^n} \Delta_\eta G(\eta -\zeta, \tau) \phi(\zeta) d\zeta 
	\right|^\frac{p}{p-1} R^{n+2} d\eta d\tau \\
	&\leq 
	C_\phi R^{n-\frac{2}{p-1}} \| f\|_X
	\to 0 \quad \mbox{ as }R\to\infty, 
\end{aligned}
\]
where $C_\phi>0$ depends only on $n$, $p$ and $\phi$.  
Since the left-hand side is bounded from below by 
$\int_{B_R(0)} |f(z)| dz$, 
we obtain $f\equiv 0$ and $X=\{0\}$ for $p<p_F$.

We next consider the case $p=p_F$. 
Let $f\in X$. We use the same $\phi$ and define 
\[
	\psi_N(x):= \frac{1}{N} \sum_{j=1}^N \phi(2^{-j} x)
	\quad (N=1,2,\ldots). 
\]
Then by \eqref{eq:Xdef} and a decomposition of the time interval, 
\[
\begin{aligned}
	&\int_{\R^n} |f(z)| |\psi_N(z)|^\frac{p}{p-1} dz 
	\leq 
	\| f\|_X 
	\int_0^\infty \int_{\R^n} 
	\left| \int_{\R^n} \Delta_y G(y-z,s) \psi_N(z) dz 
	\right|^\frac{p}{p-1} dy ds  \\
	&= \| f\|_X  \left( \int_0^4 
	+ \sum_{k=1}^N \int_{4^k}^{4^{k+1}} 
	+ \int_{4^{N+1}}^\infty \right)
	\int_{\R^n} 
	\left| \int_{\R^n} \Delta_y G(y-z,s) \psi_N(z) dz 
	\right|^\frac{p}{p-1} dy ds \\
	&=: \| f\|_X  \left( I_N^{(1)} + I_N^{(2)} + I_N^{(3)} \right). 
\end{aligned}
\]

We consider $I_N^{(2)}$. 
From the smoothing estimate for $j\leq k$ 
and the integration by parts for $j>k$, it follows that 
\[
\begin{aligned}
	&\left\| \int_{\R^n} \Delta_y G(\cdot-z,s) \psi_N(z) dz 
	\right\|_{L^\frac{p}{p-1}(\R^n)}  \\
	&\leq \frac{C}{N} \left\| 
	\sum_{j=1}^k \int_{\R^n} \Delta_y G(\cdot-z,s) \phi(2^{-j} z) dz 
	\right\|_{L^\frac{p}{p-1}(\R^n)}    \\
	&\quad 
	+ \frac{C}{N} \left\| 
	\sum_{j=k+1}^N 4^{-j} \int_{\R^n} G(\cdot-z,s) \Delta \phi(2^{-j} z) dz 
	\right\|_{L^\frac{p}{p-1}(\R^n)}  \\
	&\leq 
	\frac{C_\phi}{N} s^{-1} \sum_{j=1}^k 2^{\frac{n(p-1)}{p}j}
	+\frac{C_\phi}{N} \sum_{j=k+1}^N 2^{(\frac{n(p-1)}{p}-2)j}. 
\end{aligned}
\]
By estimating the right-hand side with $p=p_F$ and 
$s^{-1}\leq 4^{-k}$ for $4^k \leq s <4^{k+1}$, 
\[
	\left\| \int_{\R^n} \Delta_y G(\cdot-z,s) \psi_N(z) dz 
	\right\|_{L^\frac{p}{p-1}(\R^n)} 
	\leq 
	\frac{C_\phi}{N} 4^{-\frac{p-1}{p}k}. 
\]
Therefore, there exists $C_\phi>0$ depending only on $n$, $p$ and $\phi$ 
such that 
\[
\begin{aligned}
	I_N^{(2)} &\leq 
	\sum_{k=1}^N \int_{4^k}^{4^{k+1}} 
	\left( \frac{C_\phi}{N} 4^{-\frac{p-1}{p}k} \right)^\frac{p}{p-1} ds 
	\leq C_\phi N^{1-\frac{p}{p-1}}. 
\end{aligned}
\]

As for $I_N^{(1)}$ and $I_N^{(3)}$, 
by the integration by parts for $I_N^{(1)}$ 
and by the smoothing estimate for $I_N^{(3)}$, 
we can see from 
similar computations to the above that 
\[
	I_N^{(1)} \leq C_\phi N^{-\frac{p}{p-1}}, 
	\quad 
	I_N^{(3)} \leq C_\phi N^{-\frac{p}{p-1}}. 
\]
Then by Fatou's lemma with $\psi_N(x)\to1$ as $N\to\infty$, 
\[
\begin{aligned}
	\int_{\R^n} |f(z)| dz 
	&\leq 
	\liminf_{N\to\infty} 
	\int_{\R^n} |f(z)| |\psi_N(z)|^\frac{p}{p-1} dz  \\
	&\leq 
	\liminf_{N\to\infty}  
	C_\phi \| f\|_X 
	\left(  N^{1-\frac{p}{p-1}} +  N^{-\frac{p}{p-1}} \right) = 0. 
\end{aligned}
\]
We obtain $f\equiv 0$ and $X=\{0\}$ also for $p=p_F$. 
\end{proof}

In view of Lemma \ref{lem:trivial}, 
we focus on the case $p>p_F$. 
We next check that 
$\| \cdot\|_X$ is a scaling invariant norm under the scaling \eqref{eq:fscaling}.

\begin{lemma}
Let $p> p_F$ and $f\in X$. 
Then $\|f_\lambda\|_X = \|f\|_X$ for $\lambda>0$. 
\end{lemma}

\begin{proof}
Let $\phi\in C_0^\infty(\R^n)$ and 
set $\phi_\lambda(x):=\phi(\lambda x)$. 
Then by the change of variables, 
\[
\begin{aligned}
	&\int_{\R^n} |f_\lambda(z)| |\phi_\lambda(z)|^\frac{p}{p-1} dz 
	= \lambda^{\frac{2}{p-1}-n} 
	\int_{\R^n} |f(\zeta)| |\phi(\zeta)|^\frac{p}{p-1} d\zeta, \\
	&\int_0^\infty \int_{\R^n} 
	\left| \int_{\R^n} \Delta_y G(y-z,s) \phi_\lambda(z) dz 
	\right|^\frac{p}{p-1} dy ds \\
	&=
	\lambda^{\frac{2}{p-1}-n} 
	\int_0^\infty \int_{\R^n} 
	\left| \int_{\R^n} \Delta_\eta G(\eta-\zeta,\tau) 
	\phi(\zeta) d\zeta \right|^\frac{p}{p-1} d\eta d\tau. 
\end{aligned}
\]
This together with the definition of $\|\cdot\|_X$ in \eqref{eq:Xdef} 
shows the lemma. 
\end{proof}

We next show the nontriviality of $X$ for $p>p_F$ 
by examining relations 
between $\|\cdot\|_X$ and known scaling invariant norms. 
Let $p>p_F$. We define 
a Lorentz type space $L$ 
and a Morrey type space $M$ 
as subspaces of $L^\infty(\R^n)$ by 
\[
\begin{aligned}
	&L:= \left\{ 
	f\in L^\infty(\R^n);\; 
	\|f\|_{L^{q_c,\infty}(\R^n)}<\infty \right\}, 
	\quad q_c:= \frac{n(p-1)}{2}, \\
	&\|f\|_{L^{q_c,\infty}(\R^n)}:= 
	\sup_{\mu>0} \left( 
	\mu \cdot | \{ x\in \R^n;\; |f(x)| > \mu \} |^\frac{1}{q_c} \right), \\
	&M:= \left\{ 
	f\in L^\infty(\R^n);\; 
	\|f\|_{M^{1,\lambda_c}(\R^n)}<\infty \right\}, 
	\quad \lambda_c:= \frac{2}{p-1}, \\
	&\|f\|_{M^{1,\lambda_c}(\R^n)}:= 
	\sup_{x\in \R^n, r>0} r^{\lambda_c-n} \int_{B_r(x)} |f(y)| dy. 
\end{aligned}
\] 
Note that $q_c>1$ and $\lambda_c-n<0$ for $p>p_F$. 
Moreover, the above Morrey norm and Lorentz (quasi-)norm 
are invariant under the scaling \eqref{eq:fscaling}. 
We prove the following relations. 
This particularly shows that $X$ 
contains nontrivial functions when $p>p_F$. 
For example, $(1+|\cdot|)^{-2/(p-1)} \in X$.

\begin{proposition}\label{pro:rel}
Let $p>p_F$. 
Then there exist $C_1,C_2>0$ depending only on $n$ and $p$ such that 
\[
	\|f_1\|_X
	\leq C_1 \|f_1\|_{ L^{q_c,\infty}(\R^n)}, \quad 
	\| f_2\|_{M^{1,\lambda_c}(\R^n)} 
	\leq C_2 \|f_2\|_X
\]
for $f_1\in L$ and $f_2\in X$. 
Moreover, there exist nonnegative functions 
$g_1\in X$ 
and $g_2\in M$ such that 
$g_1\not \in L$ and $g_2 \not \in X$. 
In particular, $L\subsetneq X \subsetneq M$, that is, 
\[
	L^{q_c,\infty}(\R^n) 
	\cap 
	L^\infty(\R^n) 
	\subsetneq X 
	\subsetneq M^{1,\lambda_c}(\R^n) 
	\cap L^\infty(\R^n). 
\]
\end{proposition}

We prove this proposition by means of several lemmas. 
In the rest of this section, $C>0$ depends only on $n$ and $p$. 
Moreover, we write $L^{q_c,\infty}:= L^{q_c,\infty}(\R^n)$ 
and $M^{1,\lambda_c}:= M^{1,\lambda_c}(\R^n)$.

\begin{lemma}\label{lem:XL}
$\|f\|_X \leq C\|f\|_{L^{q_c,\infty}}$ for $f\in L$. 
\end{lemma}

\begin{proof}
Let $f\in L$ and let  $\phi\in C_0^\infty(\R^n)$. 
We denote by $f_*$ and $\phi_*$ 
the nonincreasing rearrangements of $|f|$ and $|\phi|$, respectively. 
Then, the rearrangement inequality and 
the characterization of the Lorentz space yield  
\[
\begin{aligned}
	&\int_{\R^n} |f(z)| |\phi(z)|^\frac{p}{p-1} dz
	\leq 
	\int_0^\infty f_*(\nu) \phi_*^\frac{p}{p-1}(\nu) d\nu 
	\leq 
	\|f\|_{L^{q_c,\infty}} 
	\int_0^\infty \nu^{-\frac{2}{n(p-1)}}
	\phi_*^\frac{p}{p-1} (\nu) d\nu \\
	&= 
	\|f\|_{L^{q_c,\infty}}
	\int_0^\infty \left[ \nu^\frac{n(p-1)-2}{np}
	\phi_*(\nu) \right]^\frac{p}{p-1}  \frac{d\nu}{\nu }
	= \|f\|_{L^{q_c,\infty}}
	\| \phi\|_{L^{\frac{np}{n(p-1)-2}, \frac{p}{p-1}}}^\frac{p}{p-1}. 
\end{aligned}
\]
Applying the Sobolev embedding into the Lorentz space 
(see \cite[TH\'EOR\`EME 8.1]{Pe66} for instance) and 
using the lifting property and the heat semigroup characterization
of homogeneous Besov spaces
(see \cite[Proposition 2.30, Theorem 2.34]{BCDbook}), 
we see that 
\[
\begin{aligned}
	\| \phi\|_{L^{\frac{np}{n(p-1)-2}, \frac{p}{p-1}}}^\frac{p}{p-1}
	&\leq  
	C \| \phi\|_{\dot{B}^{2/p}_{p/(p-1),p/(p-1)} }^\frac{p}{p-1} 
	\leq 
	C \| \Delta \phi\|_{\dot{B}^{-2(p-1)/p}_{p/(p-1),p/(p-1)} }^\frac{p}{p-1} \\
	&\leq 
	C \int_0^\infty \int_{\R^n} 
	\left| \int_{\R^n} \Delta_y G(y-z,s) \phi(z) dz \right|^\frac{p}{p-1} 
	dy ds, 
\end{aligned}
\]
where $C>0$ depends only on $n$ and $p$. 
Then the lemma follows. 
\end{proof}

\begin{lemma}\label{lem:MX}
$\| f\|_{M^{1,\lambda_c}} \leq C \|f\|_X$ for $f\in X$. 
\end{lemma}

\begin{proof}
Let $f\in X$ and let $\phi\in C_0^\infty(\R^n)$ 
satisfy $0\leq \phi\leq 1$, 
$\phi=1$ on $B_1(0)$ 
and $\phi=0$ on $\R^n\setminus B_2(0)$. 
For $x\in \R^n$ and $R>0$, set 
$\tilde \phi(y):= \phi( (y-x)/R )$. 
By \eqref{eq:Xdef} and the change of variables, 
\[
\begin{aligned}
	&\int_{B_R(x)} |f(y)| dy 
	\leq \int_{\R^n} |f(y)| 
	|\tilde \phi(y)|^\frac{p}{p-1} dy  \\
	&\leq \|f\|_X 
	\int_0^\infty \int_{\R^n} 
	\left| \int_{\R^n} \Delta_y G(y-z,s) \tilde \phi(z) dz \right|^\frac{p}{p-1} 
	dy ds \\
	&= 
	R^{n-\frac{2}{p-1}} \|f\|_X 
	\int_0^\infty \int_{\R^n} 
	\left| \int_{\R^n} \Delta_\eta G(\eta-\zeta, \tau) 
	\phi(\zeta) d\zeta \right|^\frac{p}{p-1} 
	d\eta d\tau 
\end{aligned}
\]
for $x\in \R^n$ and $R>0$. 
By \eqref{eq:forcutoff}, 
the last integral can be bounded by a constant $C_\phi>0$ 
depending only on $n$, $p$ and $\phi$. Consequently, 
\[
\begin{aligned}
	\int_{B_R(x)} |f(y)| dy 
	\leq 
	C_\phi R^{n-\lambda_c} \|f\|_X 
\end{aligned}
\]
for $x\in \R^n$ and $R>0$. 
Taking the supremum gives the desired estimate. 
\end{proof}

\begin{lemma}\label{lem:XLcounter}
There exists a nonnegative function  $g\in X$ such that 
$g\not \in L$. 
\end{lemma}

\begin{proof}
Let $e:=(1,0,\cdots,0)$ be a unit vector. 
Define a bounded function $g$ by 
\[
	g(x):= \sum_{j=1}^\infty \chi_{B_1(2^{j+3} e)}(x), 
\]
where this kind of function has been used in 
\cite[Example 3.6]{AS17} for studying vanishing properties 
of Morrey spaces. 
Then, $\|g\|_L=\infty$ by $|\{ x;\; g(x)\geq 1/2\}|=\infty$. 
We check that $\|g\|_X<\infty$. 
Let $\phi\in C_0^\infty(\R^n)$. 
For  estimating  
$\int_{\R^n}|g(z)| |\phi(z)|^{p/(p-1)} dz$ 
by using the denominator in \eqref{eq:Xdef}, we have to make 
a quantity like $\Delta_y G(y-z,s) \phi(z)$. 
To do so, we define 
\begin{equation}\label{eq:hphidef}
\begin{aligned}
	&h_\phi(x,t):= -2 \int_{\R^n} \Delta_x G(x-y,t) \phi (y) dy, \\
	&\cL[h_\phi](x):= \int_0^\infty \int_{\R^n} G(x-y,t) h_\phi(y,t) dy dt. 
\end{aligned}
\end{equation}

We claim that 
\begin{equation}\label{eq:back}
	\phi(x) = \cL[h_\phi](x), \quad x\in \R^n. 
\end{equation}
For $m\geq 2$, by Fubini's theorem 
with the semigroup property, 
\[
\begin{aligned}
	&\int_{1/m}^m \int_{\R^n} G(x-y,t) h_\phi(y,t) dy dt \\ 
	&= -2 \int_{1/m}^m \int_{\R^n} G(x-y,t) 
	\int_{\R^n} \Delta_y G(y-z,t) \phi (z) dz dy dt \\
	&= -2 \Delta_x \int_{1/m}^m \int_{\R^n} G(x-y,t) 
	\int_{\R^n}  G(y-z,t) \phi (z) dz dy dt \\
	&= -2 \Delta_x \int_{1/m}^m \int_{\R^n}  
	G(x-z,2t) \phi (z) dz  dt 
	= -2 \int_{1/m}^m \int_{\R^n}  
	\Delta_x  G(x-z,2t) \phi (z) dz  dt, 
\end{aligned}
\]
and so by $G_t=\Delta G$, 
\[
\begin{aligned}
	&\int_{1/m}^m \int_{\R^n} G(x-y,t) h_\phi(y,t) dy dt 
	= -\int_{1/m}^m \int_{\R^n}  
	\partial_t \left(  G(x-z,2t) \right)\phi (z) dz  dt \\
	&= 
	\int_{\R^n}  G(x-z,2/m) \phi (z) dz 
	- \int_{\R^n}  G(x-z,2m) \phi (z) dz 
	\to \phi (x) 
	\quad \mbox{ as }m\to \infty. 
\end{aligned}
\]
Then the claim \eqref{eq:back} holds.

By \eqref{eq:back}, 
the definition of $\cL[h_\phi]$ in \eqref{eq:hphidef} 
and Minkowski's inequality, 
\[
\begin{aligned}
	&\int_{\R^n}|g(z)| |\phi(z)|^\frac{p}{p-1} dz 
	= \int_{\R^n}|g(x)| |\cL[h_\phi](x)|^\frac{p}{p-1} dx   \\
	&\leq 
	C \int_{\R^n}|g(x)| \left| 
	\int_0^\infty \int_{\R^n} 
	G(x-y,t) h_\phi(y,t) dydt  
	\right|^\frac{p}{p-1} dx  \\
	&= 
	C \left( \int_{\R^n} \left| 
	\int_0^\infty  
	\left[ |g(x)|^\frac{p-1}{p} \int_{\R^n} 
	G(x-y,t) h_\phi(y,t) dy \right] dt  
	\right|^\frac{p}{p-1} dx 
	\right)^{\frac{p-1}{p}\cdot \frac{p}{p-1}} \\
	&\leq 
	C \left( \int_0^\infty  \left( 
	\int_{\R^n} 
	|g(x)|\left| \int_{\R^n} 
	G(x-y,t) h_\phi(y,t) dy \right|^\frac{p}{p-1} dx  
	\right)^\frac{p-1}{p} dt 
	\right)^\frac{p}{p-1}. 
\end{aligned}
\]
Then, from the definition of $h_\phi$ in \eqref{eq:hphidef}, 
the H\"older inequality for the integral in $y$ 
with $\int_{\R^n} G dy=1$  
and Fubini's theorem, it follows that 
\[
\begin{aligned}
	&\int_{\R^n} 
	|g(x)|\left| \int_{\R^n} 
	G(x-y,t) h_\phi(y,t) dy \right|^\frac{p}{p-1} dx   \\
	&= 
	2^\frac{p}{p-1} \int_{\R^n} 
	|g(x)|\left| \int_{\R^n} 
	G(x-y,t) \int_{\R^n} \Delta_y G(y-z,t) \phi(z) dz dy \right|^\frac{p}{p-1} dx  \\
	&\leq 
	2^\frac{p}{p-1} \int_{\R^n} 
	|g(x)| \int_{\R^n} 
	G(x-y,t) \left| \int_{\R^n} \Delta_y G(y-z,t) \phi(z) dz
	\right|^\frac{p}{p-1} dy  dx \\
	&=
	2^\frac{p}{p-1} \int_{\R^n} 
	\left| \int_{\R^n} \Delta_y G(y-z,t) \phi(z) dz
	\right|^\frac{p}{p-1} U_g(y,t) dy, 
\end{aligned}
\]
where 
\[
	U_g(y,t):= \int_{\R^n}  G(x-y,t) g(x)  dx. 
\]

Hence by combining the above computations, 
and then by taking the supremum in $y$ of $U_g(y,t)$
and using the H\"older inequality, we see that 
\[
\begin{aligned}
	&\int_{\R^n}|g(z)| |\phi(z)|^\frac{p}{p-1} dz \\
	&\leq 
	C \left( \int_0^\infty  \left( 
	\int_{\R^n} 
	\left| \int_{\R^n} \Delta_y G(y-z,t) \phi(z) dz \right|^\frac{p}{p-1} 
	U_g(y,t) dy 
	\right)^\frac{p-1}{p} dt 
	\right)^\frac{p}{p-1} \\
	&\leq 
	C \left( \int_0^\infty  
	\| U_g(\cdot,t)\|_{L^\infty(\R^n)}^\frac{p-1}{p} 
	\left( 
	\int_{\R^n} 
	\left| \int_{\R^n} \Delta_y G(y-z,t) \phi(z) dz \right|^\frac{p}{p-1} 
	dy 
	\right)^\frac{p-1}{p} dt 
	\right)^\frac{p}{p-1} \\
	&\leq 
	C \left( \int_0^\infty  
	\| U_g(\cdot,t)\|_{L^\infty(\R^n)}^{p-1} dt 
	\right)^\frac{1}{p-1}
	\int_0^\infty  \int_{\R^n} 
	\left| \int_{\R^n} \Delta_y G(y-z,t) \phi(z) dz \right|^\frac{p}{p-1} 
	dy dt. 
\end{aligned}
\]

In what follows, we estimate $U_g$. 
Let $x\in \R^n$ and $R>1$. 
We observe that if $B_R(x)$ intersects $m$ of the unit balls 
in $\{B_1(2^{j+3} e)\}_{j=1}^\infty$, 
then the diameter $2R$ of $B_R(x)$ 
is estimated from below by $ c 2^m$ with an absolute constant $c>0$. 
Then there exists an absolute constant $\tilde C>0$ such that 
$m\leq \tilde C \log(R+2)$ for $R>1$, 
and so there exists $C_n>0$ depending only on $n$ such that 
\[
	\int_{B_R(x)} g(y) dy \leq C_n \log(R+2) 
	\quad \mbox{ for }x\in\R^n,\; R>1. 
\]
We apply this estimate to a dyadic decomposition of $U_g$. Then, 
\[
\begin{aligned}
	&U_g(x,t) 
	\leq 
	C_n t^{-\frac{n}{2}} 
	\int_{B_{\sqrt{t}}(x)} e^{-\frac{|x-y|^2}{4t}} g dy 
	+ C_n t^{-\frac{n}{2}} 
	\sum_{k=1}^\infty 
	\int_{2^{k-1}\sqrt{t}\leq |x-y|
	\leq 2^k \sqrt{t}} e^{-\frac{|x-y|^2}{4t}} g dy \\
	&\leq 
	C_n t^{-\frac{n}{2}} \log(\sqrt{t}+2) 
	+ C_n t^{-\frac{n}{2}} \sum_{k=1}^\infty 
	e^{-4^{k-2}} \log(2^k \sqrt{t}+2) 
	\leq 
	C_n t^{-\frac{n}{2}} \log(t+2) 
\end{aligned}
\]
for $x\in \R^n$ and $t>0$, 
where $C_n>0$ depends only on $n$ and 
changes its value from line to line. 
Moreover, we see from $g\leq 1$ that 
\[
	U_g(x,t) \leq \int_{\R^n} G(x-y,t) dy =1 
	\quad 
	\mbox{ for }x\in \R^n, \; t>0. 
\]
Therefore, we can estimate that 
\[
	\int_0^\infty  
	\| U_g(\cdot,t)\|_{L^\infty(\R^n)}^{p-1} dt 
	\leq 
	\int_0^1 1^{p-1} dt 
	+ C \int_1^\infty  
	t^{-\frac{n(p-1)}{2}} \left( \log(t+2)\right)^{p-1} dt. 
\]
Since $-n(p-1)/2<-1$ for $p>p_F$, we obtain 
$\int_0^\infty  \| U_g(\cdot,t)\|_{L^\infty(\R^n)}^{p-1} dt\leq C$, 
and so 
\[
	\int_{\R^n}|g(z)| |\phi(z)|^\frac{p}{p-1} dz 
	\leq 
	C \int_0^\infty  \int_{\R^n} 
	\left| \int_{\R^n} \Delta_y G(y-z,t) \phi(z) dz \right|^\frac{p}{p-1} 
	dy dt, 
\]
where $C>0$ depends only on $n$ and $p$. 
This implies that $\|g\|_X \leq C$. 
We recall from the beginning of this proof that $\|g\|_L=\infty$. 
Then the lemma follows. 
\end{proof}

\begin{lemma}\label{lem:MXcounter}
There exists a nonnegative function  $g\in M$ such that 
$g\not \in X$. 
\end{lemma}

\begin{proof}
Similarly to Lemma \ref{lem:XLcounter}, 
we consider a sum of characteristic functions 
on countably many unit balls. 
Let $m\geq 0$ be the largest integer satisfying $m<n-(2/(p-1))$ 
and let $\theta:= n-(2/(p-1)) -m$. 
We take the centers of balls $\xi\in\R^n$ of the form 
\[
\begin{aligned}
	&\xi=(x_1,\ldots,x_m, x_{m+1}, 0,\ldots,0)
	&&\in 4\Z^m \times D_\theta \times \{0\}^{n-m-1}
	&&\mbox{ if } 0<m<n-1, \\
	&\xi=(x_{m+1}, 0,\ldots,0)
	&&\in D_\theta \times \{0\}^{n-1}
	&&\mbox{ if } m=0,\\
	&\xi=(x_1,\ldots,x_m, x_{m+1})
	&&\in 4\Z^m \times D_\theta 
	&&\mbox{ if } m=n-1, \\
\end{aligned}
\]
where $4\Z^m$ consists of the lattice points 
of the form $(4y_1,\ldots,4y_m)$ with $y_1,\ldots,y_m \in\Z$ 
and $D_\theta$ is defined by 
\[
	D_\theta:= 
	\left\{ 4\sum_{j=0}^N \eps_j 2^\frac{j}{\theta}; \; 
	N=0, 1, \ldots, \eps_j\in \{0,1\} \right\}. 
\]
For example, if $n=1$ and $p=5$, then  
$n-(2/(p-1))=1/2$, 
$m=0$, $\theta=1/2$ and 
\[
\begin{aligned}
	D_{1/2}&= 4 \{0\cdot 4^0, 1\cdot 4^0, 
	0\cdot 4^0 + 1\cdot 4^1, 
	1\cdot 4^0 + 1\cdot 4^1, 
	0\cdot 4^0 + 0\cdot 4^1 + 1 \cdot 4^2, \\
	&\qquad 
	1\cdot 4^0 + 0\cdot 4^1 + 1 \cdot 4^2, 
	0\cdot 4^0 + 1\cdot 4^1 + 1 \cdot 4^2, 
	1\cdot 4^0 + 1\cdot 4^1 + 1 \cdot 4^2, \ldots \} \\
	&=4\{ 0,1,4,5, 16,17, 20,21, 64, \ldots \}
	= \{ 0,4,16,20, 64,68, 80,84, 256, \ldots \}. 
\end{aligned}
\]
Moreover, if $n=3$ and $p=2$, 
then $n-(2/(p-1))=1$, $m=0$, $\theta=1$ and 
\[
	D_1 
	= 4\{0, 1, 2, 3, 4, 5, 6, 7, \ldots \}
	= \{0, 4, 8, 12, 16, 20, 24, 28, \ldots \}. 
\]

We denote by $\Xi$ the above set of the centers of the form 
$4\Z^m \times D_\theta \times \{0\}^{n-m-1}$
and define 
a bounded function $g$ by 
\begin{equation}\label{eq:Morreyg}
	g(x):= \sum_{\xi\in \Xi} \chi_{B_1(\xi)}(x). 
\end{equation}
Since $0<\theta\leq 1$,  
the distance between any two distinct centers in 
$\Xi$ is at least $4$. 
Thus, we observe that the balls $B_1(\xi)$ are pairwise disjoint 
and that $0\leq g(x)\leq 1$ for $x\in\R^n$. 
For example, in the case $n=3$ and $p=2$, 
we have $\Xi=D_1 \times \{0\}^2$ and 
\[
	g(x)= \sum_{\xi_1 \in D_1} \chi_{B_1( (\xi_1,0,0) )}(x)
	= \sum_{k=0}^\infty  \chi_{B_1(4k e)}(x), 
	\quad e:=(1,0,0). 
\]

By counting the centers of balls in \eqref{eq:Morreyg}, 
we show that for $r\geq1$, 
\begin{align}
	&\#( \Xi \cap B_r(x) ) \leq C r^{n-\frac{2}{p-1}} 
	&&\mbox{ for }x\in \R^n, \label{eq:upper} \\
	&\#( \Xi \cap B_r(\xi) ) \geq C^{-1} r^{n-\frac{2}{p-1}} 
	&&\mbox{ for }\xi\in \Xi, \label{eq:lower} 
\end{align}
where $\# A$ is the number of elements of $A$ 
and $C>1$ depends only on $n$. 
To show these estimates, we first observe $D_\theta$ 
and derive one dimensional estimates 
\begin{align}
	&\#(D_\theta \cap (a-r, a+r)) \leq 12 r^{\theta} 
	&&\mbox{ for }a \in \R,\;r\geq1, \label{eq:upper1} \\
	&\#(D_\theta \cap (a-r, a+r)) \geq \frac{1}{8} r^{\theta} 
	&&\mbox{ for }a \in D_\theta, \;r\geq1. \label{eq:lower1} 
\end{align}
For $1\leq r<8$, since the distance between two distinct points 
in $D_\theta$ is at least $4$, we see that 
$\#(D_\theta \cap (a-r, a+r)) \leq 4 \leq 12 r^\theta$ 
for $a\in \R$. 
As for the lower bound, we have 
$\#(D_\theta \cap (a-r, a+r)) \geq  1 \geq r^\theta/8$ 
for $a\in D_\theta$. 
Therefore, \eqref{eq:upper1} and \eqref{eq:lower1} hold 
for $1\leq r<8$.

Let $r\geq8$. We fix $k\geq0$ 
such that $8\cdot 2^{k/\theta} \leq r < 8\cdot 2^{(k+1)/\theta}$. 
For $\eta \in 2^{(k+1)/\theta} D_\theta$, 
define a block $P_k(\eta)$ by 
\[
	P_k(\eta):= \left\{ 
	\eta + 4\sum_{j=0}^k \eps_j 2^\frac{j}{\theta};\; 
	\eps_0,\ldots,\eps_k\in\{0,1\}
	\right\}. 
\]
We remark that each of the blocks $P_k(\eta)$ is made by 
dividing $D_\theta$ into subsets by fixing 
$\eps_{k+1}$, $\eps_{k+2}$, \ldots 
and varying $\eps_0$, \ldots, $\eps_k$ in $\{ 0,1\}$. 
For  $k$ fixed as above,  we have 
\begin{equation}\label{eq:decomdis}
	D_\theta = \bigcup_{\eta\in 2^\frac{k+1}{\theta} D_\theta} 
	P_k(\eta) \quad (\mbox{disjoint union}). 
\end{equation}
Each of the blocks $P_k(\eta)$ has $2^{k+1}$ points, 
has the minimum point $\eta$ and admits the diameter bound 
\begin{equation}\label{eq:diamupper}
	\diam P_k(\eta) = 4\sum_{j=0}^k 2^\frac{j}{\theta} 
	=4\frac{2^\frac{k+1}{\theta}-1}{2^\frac{1}{\theta}-1} 
	< 8\cdot 2^\frac{k}{\theta} \leq r. 
\end{equation}
Moreover, for $\eta, \eta'\in 2^{(k+1)/\theta} D_\theta$ with 
$\eta\neq \eta'$, we have 
\begin{equation}\label{eq:space}
	|\eta-\eta'| \geq 4\cdot 2^\frac{k+1}{\theta} \geq \frac{r}{2}. 
\end{equation}
Therefore, for $a\in \R$, 
if $P_k(\eta)\cap (a-r, a+r)\neq \emptyset$, 
then $\eta$ belongs to the interval $(a-2r, a+r)$ with length $3r$. 
This together with \eqref{eq:space}, 
$3r/(r/2)=6$, $\# P_k(\eta)=2^{k+1}$ and $2^k \leq 8\cdot 2^{k/\theta} \leq r$ 
implies the following upper bound: 
\[
	\#(D_\theta \cap (a-r, a+r)) \leq 
	6 \cdot 2^{k+1} \leq 12 r^\theta 
	\quad 
	\mbox{ for }a \in \R,\;r\geq8. 
\]
As for the lower bound, we use the same $k$ fixed as above. 
Let $a\in D_\theta$. By \eqref{eq:decomdis}, 
we can take $P_k(\eta_a)$ such that $a\in P_k(\eta_a)$. 
We remark that $\diam P_k(\eta_a)<r$ by \eqref{eq:diamupper}, 
and so $P_k(\eta_a) \subset(a-r,a+r)$. 
Therefore, 
\[
	\#(D_\theta \cap (a-r, a+r)) 
	\geq \#P_k(\eta_a) = 2^{k+1} 
	> 8^{-\theta} r^\theta  \geq \frac{1}{8} r^\theta
	\quad 
	\mbox{ for }a \in D_\theta,\;r\geq8. 
\]
Therefore, \eqref{eq:upper1} and \eqref{eq:lower1} 
hold for $r\geq 8$. 
Since the case $1\leq r< 8$ is already handled, 
\eqref{eq:upper1} and \eqref{eq:lower1} also hold for all $r\geq 1$.

We use $4\Z^m \times D_\theta \times \{0\}^{n-m-1}$ 
to show \eqref{eq:upper} and \eqref{eq:lower}. 
By counting the lattice points and 
by using \eqref{eq:upper1} and $m+\theta=n-(2/(p-1))$, 
for $x\in\R^n$ and $r\geq1$, 
\[
	\#( \Xi \cap B_r(x) ) \leq C r^m \times r^\theta = C r^{n-\frac{2}{p-1}}, 
\]
where $C>0$ depends only on $n$. 
On the other hand, by \eqref{eq:lower1} and a geometric observation, 
for $\xi\in\Xi$ and $r\geq 2\sqrt{n}$, 
\[
	\#( \Xi \cap B_r(\xi) ) 
	\geq \frac{1}{8\cdot 4^m} 
	\left( \frac{r}{2\sqrt{n}} \right)^{m+\theta} 
	\geq \frac{1}{8\cdot 4^n ({2\sqrt{n}})^n} 
	r^{n-\frac{2}{p-1}}. 
\]
For $\xi\in\Xi$ and $1\leq r< 2\sqrt{n}$, 
we have $\#( \Xi \cap B_r(\xi) ) \geq 1 \geq (2\sqrt{n})^{-n} r^{n-(2/(p-1))}$. 
Hence we obtain \eqref{eq:upper} and \eqref{eq:lower}.

We return to the estimates concerning $g$ in \eqref{eq:Morreyg}. 
From \eqref{eq:upper}, it follows that 
\[
	\int_{B_r(x)} g(y) dy
	\leq |B_1(0)| \times \#( \Xi \cap B_{r+1}(x) )
	\leq C r^{n-\frac{2}{p-1}} 
\]
for $r\geq 1$ and $x\in\R^n$. 
As for the case $0<r<1$, by $g\leq 1$, we also have 
\[
	\int_{B_r(x)} g(y) dy \leq C r^n \leq 
	C r^{n-\frac{2}{p-1}}. 
\]
Hence we obtain $g\in M$. 
In the rest of this proof, we show $g\not \in X$. 

For $R>16$, we define 
\[
	\psi_R(x):= \int_{16}^{R^2} 
	t^{\frac{p}{p-1}-2} \int_{B_{2R}(0)} 
	G(x-y,2t) g(y) dy dt. 
\]
We take $\chi\in C_0^\infty(\R^n)$ 
such that $0\leq \chi\leq 1$, 
$\chi(x)=1$ on $B_1(0)$ and 
$\chi(x)=0$ on $\R^n\setminus B_2(0)$.  
For $L>R$, set $\chi_L(x):= \chi(x/L)$ and 
$\phi_{L,R}(x):= \chi_L(x) \psi_R(x)$. 
We note that $\phi_{L,R}\in C_0^\infty(\R^n)$. 
Let us estimate 
\begin{equation}\label{eq:fractional}
	\frac{ \displaystyle 
	\int_{\R^n} |g(y)| |\phi_{L,R}(y)|^\frac{p}{p-1} dy}
	{ \displaystyle 
	\int_0^\infty \int_{\R^n} 
	\left| \int_{\R^n} \Delta_y G(y-z,s) \phi_{L,R}(z) dz 
	\right|^\frac{p}{p-1} dy ds }. 
\end{equation}

By considering the support of $\chi_L$ with $L>R$ 
and using $B_{\sqrt{t}}(y)\subset B_{2R}(0)$ 
for $y\in B_R(0)$ and $t\leq R^2$, we see that 
\[
\begin{aligned}
	&\int_{\R^n} |g(y)| |\phi_{L,R}(y)|^\frac{p}{p-1} dy \\
	&\geq 
	\int_{B_R(0)} g(y) \left| 
	\int_{16}^{R^2} 
	t^{\frac{p}{p-1}-2} \int_{B_{2R}(0)} 
	G(y-z,2t) g(z) dz dt
	\right|^\frac{p}{p-1} dy \\
	&\geq 
	C^{-1} \int_{B_R(0)\cap \Omega} g(y) \left| 
	\int_{16}^{R^2} 
	t^{\frac{p}{p-1}-2-\frac{n}{2}} \int_{B_{\sqrt{t}}(y)} 
	g(z) dz dt
	\right|^\frac{p}{p-1} dy, 
\end{aligned}
\]
where $\Omega:=\{ x\in \R^n; \; g(x)>0\}$ and 
$C>0$ depends only on $n$ and $p$.

For $y\in B_R(0)\cap \Omega$, 
we can choose $\xi_y\in \Xi$ such that 
$y\in B_1(\xi_y)$ by the definition of $g$ in \eqref{eq:Morreyg}. 
Let $t>16$. 
We observe that 
if $\xi\in \Xi \cap B_{\sqrt{t}/2}(\xi_y)$, then 
$B_1(\xi) \subset B_{\sqrt{t}}(y)$. 
Indeed, for $y'\in B_1(\xi)$,  
\[
	|y'-y|\leq |y'-\xi| + |\xi-\xi_y| + |\xi_y - y|
	\leq 
	1+\frac{\sqrt{t}}{2}+1< \sqrt{t}. 
\]
Therefore, by the counting estimate \eqref{eq:lower} 
with $r=\sqrt{t}/2$ and $t\geq 16$, 
\[
	\int_{B_{\sqrt{t}}(y)} g(z) dz
	\geq C^{-1} \left( \frac{\sqrt{t}}{2} \right)^{n-\frac{2}{p-1}}
	\geq C^{-1} t^{\frac{n}{2}-\frac{1}{p-1}}, 
\]
and so by $0\in \Xi\cap \Omega$ and \eqref{eq:lower} with $r=R/2$,
\begin{equation}\label{eq:lowRR}
\begin{aligned}
	&\int_{\R^n} |g(y)| |\phi_{L,R}(y)|^\frac{p}{p-1} dy 
	\geq 
	C^{-1} \int_{B_R(0)\cap \Omega} g(y) \left| 
	\int_{16}^{R^2} t^{-1} dt
	\right|^\frac{p}{p-1} dy  \\
	&\geq 
	C^{-1} (\log R)^\frac{p}{p-1} 
	\int_{B_R(0)\cap \Omega} g(y)  dy 
	\geq 
	C^{-1} R^{n-\frac{2}{p-1}}  (\log R)^\frac{p}{p-1}
\end{aligned}
\end{equation}
for $R>16$ and $L>R$, where 
$C>0$ depends only on $n$ and $p$.

We estimate the denominator in \eqref{eq:fractional} 
with $\phi_{L,R}$ replaced by $\psi_R(y)$. 
By the semigroup property, we note that 
\[
\begin{aligned}
	\psi_R(y)&= \int_{16}^{R^2} 
	t^{\frac{p}{p-1}-2} \int_{B_{2R}(0)} 
	G(y-z,2t) g(z) dz dt \\
	&= \int_{16}^{R^2} 
	t^{\frac{p}{p-1}-2} \int_{B_{2R}(0)} g(z)
	\int_{\R^n} G(y-\eta,t) G(\eta-z,t) d\eta  dz dt. 
\end{aligned}
\]
Therefore, Fubini's theorem gives 
\[
\begin{aligned}
	&\int_0^\infty \int_{\R^n} 
	\left| \int_{\R^n} \Delta_x G(x-y,s) \psi_R(y) dy 
	\right|^\frac{p}{p-1} dx ds \\
	&=
	\int_0^\infty \int_{\R^n} 
	\left| \Delta_x \int_{\R^n}  G(x-y,s) \psi_R(y) dy 
	\right|^\frac{p}{p-1} dx ds \\
	&= 
	\int_0^\infty \int_{\R^n} 
	\left|  \Delta_x \int_{\R^n}  
	\int_{16}^{R^2} 
	t^{\frac{p}{p-1}-2} \int_{B_{2R}(0)} G(\eta-z,t)  g(z) \right. \\
	&\quad \times \left. \int_{\R^n}  G(x-y,s) G(y-\eta,t)  dy 
	 dz dt
	 d\eta 
	\right|^\frac{p}{p-1} dx ds. 
\end{aligned}
\]
Again by the semigroup property, 
the expression inside the absolute value in the right-hand side 
is equal to 
\[
	\int_{16}^{R^2} t^{\frac{p}{p-1}-2} 
	\int_{\R^n}  
	\Delta_x G(x-\eta,s+t) 
	\int_{B_{2R}(0)} G(\eta-z,t)  g(z) 
	dz d\eta dt. 
\]

From the above computations with Minkowski's inequality and 
the smoothing estimate, it follows that 
\begin{equation}\label{eq:pp1delcalc}
\begin{aligned}
	&\int_0^\infty \int_{\R^n} 
	\left| \int_{\R^n} \Delta_x G(x-y,s) \psi_R(y) dy 
	\right|^\frac{p}{p-1} dx ds \\
	&\leq  
	C \int_0^\infty   
	\left( \int_{16}^{R^2} \frac{t^{\frac{p}{p-1}-2}}{s+t}  
	\left\| \int_{B_{2R}(0)} G(\cdot-z,t)  g(z) 
	dz \right\|_{L^\frac{p}{p-1}} dt
	\right)^\frac{p}{p-1}  ds. 
\end{aligned}
\end{equation}
By Fubini's theorem and $\int_{\R^n} G dx=1$, 
\[
\begin{aligned}
	&\left\| \int_{B_{2R}(0)} G(\cdot-z,t)  g(z) dz 
	\right\|_{L^\frac{p}{p-1}}^\frac{p}{p-1} \\
	&= 
	\int_{\R^n} 
	\left( \int_{B_{2R}(0)} G(x-z,t)  g(z) dz 
	\right)^\frac{1}{p-1} 
	\int_{B_{2R}(0)} G(x-z,t)  g(z) dz dx \\
	&\leq 
	\left( \sup_{x\in\R^n} \int_{\R^n} G(x-z,t)  g(z) dz 
	\right)^\frac{1}{p-1} 
	\int_{B_{2R}(0)} g(z) dz. 
\end{aligned}
\]
We see from \eqref{eq:upper} that 
$\int_{B_{2R}(0)} g(z) dz \leq C R^{n-2/(p-1)}$ and that 
\[
\begin{aligned}
	&\int_{\R^n} G(x-z,t)  g(z) dz  \\
	&\leq C t^{-\frac{n}{2}} 
	\int_{|x-z|\leq \sqrt{t}} g(z) dz 
	+ C t^{-\frac{n}{2}}
	\sum_{k=0}^\infty \int_{2^k\sqrt{t} \leq |x-z|\leq 2^{k+1}\sqrt{t}}
	e^{-\frac{|x-z|^2}{4t}} g(z) dz  \\
	&\leq 
	C t^{-\frac{n}{2}} 
	\int_{|x-z|\leq \sqrt{t}} g(z) dz 
	+ C t^{-\frac{n}{2}}
	\sum_{k=0}^\infty 
	e^{-4^{k-1} }
	\int_{|x-z|\leq 2^{k+1}\sqrt{t}} g(z) dz 
	\leq 
	C t^{-\frac{1}{p-1}}. 
\end{aligned}
\]
Thus, 
\[
	\left\| \int_{B_{2R}(0)} G(\cdot-z,t)  g(z) dz 
	\right\|_{L^\frac{p}{p-1}}
	\leq 
	\left( C t^{-\frac{1}{(p-1)^2}}   R^{n-\frac{2}{p-1}}  \right)^\frac{p-1}{p}
	\leq 
	C t^{-\frac{1}{p(p-1)}}   R^{\frac{n(p-1)}{p}-\frac{2}{p}}. 
\]

The estimate of $L^{p/(p-1)}$-norm with \eqref{eq:pp1delcalc} 
and direct computations shows that 
\[
\begin{aligned}
	&\int_0^\infty \int_{\R^n} 
	\left| \int_{\R^n} \Delta_x G(x-y,s) \psi_R(y) dy 
	\right|^\frac{p}{p-1} dx ds \\
	&\leq  
	C \int_0^\infty   
	\left( \int_{16}^{R^2} \frac{1}{s+t}  
	t^{\frac{p}{p-1}-2-\frac{1}{p(p-1)}} 
	  R^{\frac{n(p-1)}{p}-\frac{2}{p}}  dt
	\right)^\frac{p}{p-1}  ds \\
	&\leq
	C R^{n-\frac{2}{p-1}} 
	\left( \int_0^{16} + \int_{16}^{R^2} + \int_{R^2}^\infty 
	\right) 
	\left( \int_{16}^{R^2} \frac{1}{s+t}  
	t^{-\frac{p-1}{p}}  dt
	\right)^\frac{p}{p-1}  ds \\
	&\leq 
	C R^{n-\frac{2}{p-1}} ( 1+ \log R + 1 )
	\leq 
	C R^{n-\frac{2}{p-1}}  \log R. 
\end{aligned}
\]
This together with the lower estimate \eqref{eq:lowRR} also shows that,  
by choosing a sufficiently large $L>R$ for each $R$, 
the fraction in \eqref{eq:fractional} can be estimated as 
\[
\begin{aligned}
	\frac{ \displaystyle 
	\int_{\R^n} |g(y)| |\phi_{L,R}(y)|^\frac{p}{p-1} dy}
	{ \displaystyle 
	\int_0^\infty \int_{\R^n} 
	\left| \int_{\R^n} \Delta_y G(y-z,s) \phi_{L,R}(z) dz 
	\right|^\frac{p}{p-1} dy ds } 
	&\geq
	C^{-1} (\log R)^{\frac{p}{p-1}-1} \to \infty
\end{aligned}
\]
as $R\to\infty$. 
Hence $g\not \in X$ and the lemma follows. 
\end{proof}

We are now in a position to prove Proposition \ref{pro:rel}. 

\begin{proof}[Proof of Proposition \ref{pro:rel}]
This proposition immediately follows from 
Lemmas \ref{lem:XL}, \ref{lem:MX}, \ref{lem:XLcounter} 
and \ref{lem:MXcounter}. 
The proof is complete. 
\end{proof}

\section{Estimates of global-in-time solutions}\label{sec:gex}
We give an estimate for global-in-time solutions.

\begin{proposition}\label{pro:esti}
Let $n\geq1$ and $p>p_F$. 
If $u$ is a global-in-time classical solution with 
a nonnegative initial datum $u_0\in L^\infty(\R^n)$, then 
$u_0$ must satisfy 
\[
	\|u_0\|_X \leq 2^\frac{p}{p-1} (p-1)^{-\frac{1}{p-1}}. 
\]
\end{proposition}

To show Proposition \ref{pro:esti}, 
we test \eqref{eq:main} 
by $v_\eps^{p/(p-1)} \chi_R$, 
where $v_\eps$ and $\chi$ are given as follows. 
Let $\eps>0$ and $h(x,t)\geq0$ with $h\in C_0^\infty(\R^n\times(0,\infty))$. 
We choose $T>0$ so large that $\supp h \subset \R^n\times(0,T)$. 
The function $v_\eps$ takes the form 
\begin{equation}\label{eq:vepsvzetadefs}
\left\{ 
\begin{aligned}
	&v_\eps(x,t):= v(x,t) + \eps \zeta(x,t), &&x\in \R^n, \; t\in (0,T), \\
	&v(x,t):= \int_t^T \int_{\R^n} G(x-y,s-t) h(y,s) dy ds, 
	&&x\in \R^n, \; t\in (0,T), \\
	&\zeta(x,t):= \int_{\R^n} 
	G(x-y,T+1-t) e^{-|y|^2} dy, 
	&&x\in \R^n, \; t\in (0,T), 
\end{aligned}
\right. 
\end{equation}
where $G(x,t)= (4\pi t)^{-n/2} e^{-|x|^2/(4t)}$. 
Moreover, $\chi_R$ is taken so that 
$\chi_R(x):= \chi(x/R)$ for $R>1$, where 
$\chi\in C_0^\infty(\R^n)$, 
$0\leq \chi\leq 1$, 
$\chi(x)=1$ for $|x|\leq 1$ and 
$\chi(x)=0$ for $|x|\geq 2$. 
We note that $\zeta$ satisfies 
the backward heat equation 
\begin{equation}\label{eq:zetaeq}
	-\zeta_t -\Delta \zeta =0, \quad x\in \R^n, \; t\in (0,T)
\end{equation}
and that $v$ satisfies the inhomogeneous backward heat equation 
\begin{equation}\label{eq:veq}
\left\{
\begin{aligned}
	&-\partial_t v -\Delta v = h(x,t), &&x\in \R^n, \; t\in(0,T), \\
	&v(x,T)=0, &&x\in \R^n, \\
	&v(x,0)=\cL[h](x), &&x\in \R^n, \\
\end{aligned}
\right.
\end{equation}
where we set 
\begin{equation}\label{eq:cLdef}
	\cL[h](x):= \int_0^\infty \int_{\R^n} G(x-y,s) h(y,s) dy ds. 
\end{equation}
The idea of using the backward heat equation is due to 
Baras--Pierre \cite{BP85}. 

Testing the equation by $v_\eps^{p/(p-1)} \chi_R$ gives the following estimate:

\begin{lemma}\label{lem:test}
Let $u$ be a global-in-time classical solution of \eqref{eq:main} 
with a nonnegative initial datum $u_0\in L^\infty(\R^n)$. 
Then, for $0<\delta<T$, 
\[
\begin{aligned}
	&\int_{\R^n} u(x,\delta) v_\eps^\frac{p}{p-1}(x,\delta) \chi_R(x) dx \\
	&\leq 
	(p-1)^{-\frac{1}{p-1}} \int_\delta^T \int_{\R^n} h^\frac{p}{p-1}(x,t) dxdt  
	+\int_{\R^n} u(x,T) v_\eps^\frac{p}{p-1}(x,T) dx \\
	&\quad 
	- \int_\delta^T \int_{\R^n} 
	u\left( 2\nabla (v_\eps^\frac{p}{p-1}) \cdot \nabla \chi_R 
	+ v_\eps^\frac{p}{p-1} \Delta \chi_R \right) dxdt. 
\end{aligned}
\]
\end{lemma}

\begin{proof}
We multiply $u_t-\Delta u =u^p$ by $v_\eps^{p/(p-1)} \chi_R$ 
and integrate it. Then, 
\[
	\int_\delta^T \int (u_t-\Delta u) v_\eps^\frac{p}{p-1} \chi_R 
	= \int_\delta^T \int u^p v_\eps^\frac{p}{p-1} \chi_R, 
\]
where we abbreviate $dxdt$ and $\int(\cdots):=\int_{\R^n}(\cdots)$. 
By applying the integration by parts to the left-hand side, 
\[
\begin{aligned}
	&\int u(x,T) v_\eps^\frac{p}{p-1}(x,T) \chi_R 
	- \int u(x,\delta) v_\eps^\frac{p}{p-1}(x,\delta) \chi_R \\
	&+\frac{p}{p-1}  \int_\delta^T \int u v_\eps^\frac{1}{p-1} 
	(-\partial_t v_\eps - \Delta v_\eps) \chi_R 
	- \frac{p}{(p-1)^2} \int_\delta^T \int 
	u v_\eps^\frac{2-p}{p-1} |\nabla v_\eps|^2 \chi_R \\
	&- \int_\delta^T \int 
	u\left( 2\nabla (v_\eps^\frac{p}{p-1}) \cdot \nabla \chi_R 
	+ v_\eps^\frac{p}{p-1} \Delta \chi_R \right) 
	= \int_\delta^T \int u^p v_\eps^\frac{p}{p-1} \chi_R. 
\end{aligned}
\]
Here we remark that our choice of $v_\eps$ guarantees that $v_\eps>0$, and so 
$v_\eps^{(2-p)/(p-1)}$ is well-defined even when $p>2$.

Since $-\partial_t v_\eps-\Delta v_\eps = h$ 
by \eqref{eq:zetaeq} and \eqref{eq:veq}, 
the fourth term in the left-hand side is nonpositive, 
$u(x,T) v_\eps^{p/(p-1)}(x,T)\geq0$  and $\chi_R\leq 1$, 
we see that 
\[
\begin{aligned}
	&\int u(x,\delta) v_\eps^\frac{p}{p-1}(x,\delta) \chi_R
	\leq  \int_\delta^T \int 
	\left( \frac{p}{p-1}  u v_\eps^\frac{1}{p-1} h 
	- u^p v_\eps^\frac{p}{p-1} \right)\chi_R
	\\
	&+\int u(x,T) v_\eps^\frac{p}{p-1}(x,T) 
	- \int_\delta^T \int_{\R^n} 
	u\left( 2\nabla (v_\eps^\frac{p}{p-1}) \cdot \nabla \chi_R 
	+ v_\eps^\frac{p}{p-1} \Delta \chi_R \right). 
\end{aligned}
\]
Fix $x\in \R^n$ and $\delta<t<T$. Then by $1+(1/(p-1))=p/(p-1)$, 
\[
\begin{aligned}
	&\frac{p}{p-1}  u(x,t) v_\eps^\frac{1}{p-1}(x,t) h(x,t) 
	- u^p(x,t) v_\eps^\frac{p}{p-1}(x,t) 
	\leq 
	\sup_{z\geq0} 
	\left( \frac{p}{p-1}  h(x,t) z - z^p \right) \\
	&= \frac{p}{p-1} h(x,t) \left( \frac{h(x,t)}{p-1} \right)^\frac{1}{p-1} 
	- \left( \left( \frac{h(x,t)}{p-1} \right)^\frac{1}{p-1} \right)^p 
	= (p-1)^{-\frac{1}{p-1}} h^\frac{p}{p-1} (x,t). 
\end{aligned}
\]
Then the lemma follows. 
\end{proof}

By taking the limits of the parameters in Lemma \ref{lem:test}, 
we obtain the following:

\begin{lemma}\label{lem:lim}
Let $u$ be a global-in-time classical solution of \eqref{eq:main} 
with a nonnegative initial datum $u_0\in L^\infty(\R^n)$. 
Then, 
\begin{equation}\label{eq:hsmooth}
	\int_{\R^n} u_0(x) \cL[h]^\frac{p}{p-1}(x) dx 
	\leq 
	(p-1)^{-\frac{1}{p-1}} \int_0^\infty \int_{\R^n} h^\frac{p}{p-1}(x,t) dxdt 
\end{equation}
for $h(x,t)\geq0$ with $h\in C_0^\infty(\R^n\times(0,\infty))$, 
where $\cL$ is defined by \eqref{eq:cLdef}. 
Moreover, \eqref{eq:hsmooth} also holds 
for $h(x,t)\geq0$ with $h\in L^{p/(p-1)}(\R^n\times(0,\infty))$. 
\end{lemma}

\begin{proof}
By approximation, 
it suffices to consider the case 
$0\leq h\in C_0^\infty(\R^n\times(0,\infty))$. 
Since $u$ is a global-in-time classical solution 
with an initial datum $u_0\in L^\infty(\R^n)$, 
we have $u\in L^\infty(\R^n\times(0,T))$. 
Moreover, we can check that 
\[
\begin{aligned}
	&v_\eps\in L^\infty(\delta,T; L^1(\R^n)) 
	\cap L^\infty(\R^n\times(\delta,T)), 
	\quad 
	\nabla v_\eps \in L^\infty(\delta,T; L^1(\R^n)), \\
	&|\nabla \chi_R|\leq C R^{-1}, \quad 
	|\Delta \chi_R|\leq C R^{-2}, 
\end{aligned}
\]
where $C>0$ is an absolute constant. 
Then, letting $R\to\infty$ in the inequality 
obtained in Lemma \ref{lem:test} yields 
\[
\begin{aligned}
	&\int_{\R^n} u(x,\delta) v_\eps^\frac{p}{p-1}(x,\delta) dx \\
	&\leq 
	(p-1)^{-\frac{1}{p-1}} \int_\delta^T \int_{\R^n} h^\frac{p}{p-1}(x,t) dxdt  
	+\int_{\R^n} u(x,T) v_\eps^\frac{p}{p-1}(x,T) dx. 
\end{aligned}
\]
Letting $\delta\to0$ with $v_\eps= v + \eps \zeta$  gives 
\[
\begin{aligned}
	&\int_{\R^n} u_0(x) \left( 
	v(x,0) + \eps \zeta(x,0) 
	\right)^\frac{p}{p-1}dx \\
	&\leq 
	(p-1)^{-\frac{1}{p-1}} \int_0^T \int_{\R^n} h^\frac{p}{p-1}(x,t) dxdt  
	+\int_{\R^n} u(x,T) \left( v(x,T) + \eps \zeta(x,T)  
	\right)^\frac{p}{p-1} dx. 
\end{aligned}
\]

By \eqref{eq:vepsvzetadefs} for $\zeta$ and by 
$v(x,0)=\cL[h](x)$ 
and $v(x,T)=0$ as in \eqref{eq:veq}, we see that 
\[
\begin{aligned}
	&\int_{\R^n} u_0(x) \left( 
	\cL[h](x) + \eps \int_{\R^n} 
	G(x-y,T+1) e^{-|y|^2} dy
	\right)^\frac{p}{p-1}dx \\
	&\leq 
	(p-1)^{-\frac{1}{p-1}} \int_0^T \int_{\R^n} h^\frac{p}{p-1} dxdt  \\
	&\quad 
	+\int_{\R^n} u(x,T) \left( 
	\eps  \int_{\R^n} G(x-y,1) e^{-|y|^2} dy
	\right)^\frac{p}{p-1} dx. 
\end{aligned}
\]
Since the left-hand side converges to 
the left-hand side in \eqref{eq:hsmooth} 
as $\eps\to0$, 
it suffices to check that the second term in the right-hand side 
converges to $0$. 
Then, the smoothing estimate shows that 
\[
\begin{aligned}
	&\int_{\R^n} u(x,T) \left( 
	\eps  \int_{\R^n} G(x-y,1) e^{-|y|^2} dy
	\right)^\frac{p}{p-1} dx \\
	&\leq 
	\eps^\frac{p}{p-1}
	\| u(\cdot,T)\|_{L^\infty(\R^n)}
	\| e^{-|\cdot|^2} \|_{L^\frac{p}{p-1}(\R^n)}^\frac{p}{p-1} \to 0
	\quad \mbox{ as } \eps\to 0. 
\end{aligned}
\]
Hence we obtain the desired estimate \eqref{eq:hsmooth}. 
\end{proof}

Based on Lemma \ref{lem:lim}, we prove Proposition \ref{pro:esti}.

\begin{proof}[Proof of Proposition \ref{pro:esti}]
For $\phi\in C_0^\infty(\R^n)$, we define $h_\phi$ 
and $\cL[h_\phi]$ as in \eqref{eq:hphidef}. 
Then by \eqref{eq:back} and Lemma \ref{lem:lim}, 
\[
\begin{aligned}
	&\int_{\R^n} |u_0(x)| |\phi(x)|^\frac{p}{p-1} dx 
	= \int_{\R^n} |u_0(x)| |\cL[h_\phi](x)|^\frac{p}{p-1} dx \\
	&\leq 
	\int_{\R^n} |u_0(x)| ( \cL[|h_\phi|](x) )^\frac{p}{p-1} dx 
	\leq 
	(p-1)^{-\frac{1}{p-1}} 
	\int_0^\infty \int_{\R^n} |h_\phi|^\frac{p}{p-1}(x,t) dxdt, 
\end{aligned}
\]
and so we obtain 
\[
\begin{aligned}
	&\int_{\R^n} |u_0(x)| |\phi(x)|^\frac{p}{p-1} dx  \\
	&\leq 
	2^\frac{p}{p-1} (p-1)^{-\frac{1}{p-1}} 
	\int_0^\infty \int_{\R^n} 
	\left| 
	\int_{\R^n} \Delta_x G(x-y,t) \phi (y) dy
	\right|^\frac{p}{p-1} dxdt. 
\end{aligned}
\]
Therefore, the definition of $X$-norm in \eqref{eq:Xdef} yields 
the desired estimate for $\|u_0\|_X$. 
The proof is complete. 
\end{proof}

\section{Existence of global-in-time solutions}\label{sec:ex}
We prove the existence of a global-in-time solution.

\begin{proposition}\label{pro:ex}
Let $n\geq1$ and $p>p_F$. 
Then there exists $c>0$ depending only on $n$ and $p$ 
such that the following holds: 
If a nonnegative function 
$u_0\in L^\infty(\R^n)$ satisfies 
$\|u_0\|_X\leq c$, then \eqref{eq:main} admits 
a unique global-in-time classical solution $u_*$. 
Moreover, there exists $C>0$ depending only on $n$ and $p$ 
such that 
\[
	u_*(x,t)\leq C \|u_0\|_X t^{-\frac{1}{p-1}} 
	\quad \mbox{ for } t>0. 
\]
\end{proposition}

Throughout this section, 
we always assume that $u_0\in L^\infty(\R^n)$ satisfies 
$\|u_0\|_X<\infty$, $u_0\geq0$ and $u_0\not \equiv 0$. 
We note from 
Proposition \ref{pro:rel} and 
Lemma \ref{lem:trivial}  that 
such a nontrivial initial datum exists for $p>p_F$ and 
does not exist for $1<p\leq p_F$.

We show Proposition \ref{pro:ex} by constructing a supersolution $\overline{u}$ 
in the sense that 
\begin{align}
	&\label{eq:inteq} 
	\overline{u}(x,t) 
	\geq \Phi[\overline{u}](x,t) 
	:= U(x,t) + \int_0^t \int_{\R^n} G(x-y,t-s) \overline{u}^p(y,s) dyds, \\
	&\label{eq:Udef}
	U(x,t):= \int_{\R^n} G(x-y,t) u_0(y) dy, 
\end{align}
where $G(x,t):= (4\pi t)^{-n/2} e^{-|x|^2/(4t)}$ 
is the heat kernel on $\R^n$. 
The idea of using supersolutions of the integral equation 
is due to Robinson--Sier{\.z}{\k{e}}ga \cite{RS13} 
and Hisa--Ishige \cite{HI18}. 
In this paper, we take a supersolution of the form 
\begin{equation} \label{eq:super}
	\overline{u}(x,t):= 
	U(x,t) + 2^p \int_0^t \int_{\R^n} G(x-y,t-s) 
	U^p (y,s) dyds. 
\end{equation}
In what follows, we prove that this $\overline{u}$ 
is indeed a supersolution 
provided that $\|u_0\|_X$ is small enough.  
To prove it, we give pointwise estimates of $\overline{u}$. 
Let us split the time integral in \eqref{eq:super} 
into $\int_0^t = \int_0^{t/2} + \int_{t/2}^t$
and estimate the first one. 
In the rest of this section, 
constants $C>0$ depend only on $n$ and $p$.

\begin{lemma}\label{lem:former}
For $x\in \R^n$ and $t>0$, 
\[
	\int_0^{t/2} \int_{\R^n} G(x-y,t-s) 
	U^p (y,s) dyds 
	\leq 
	C \int_0^\infty \int_{\R^n} 
	\left( V(x,y,t,s)
	\right)^p 
	dyds, 
\]
where $V$ is a positive function given by 
\begin{equation}\label{eq:Fdef}
	V=V(x,y,t,s):= 
	\int_{\R^n}
	G^\frac{1}{p} (x-z,t) G\left( y-z, s \right)
	u_0(z) dz. 
\end{equation}
\end{lemma}

\begin{proof}
By \eqref{eq:Udef}, we have 
\[
	G(x-y,t-s) U^p (y,s)
	= \left( \int_{\R^n} G^\frac{1}{p} (x-y,t-s) G(y-z,s) u_0(z) dz \right)^p. 
\]
For the integrand, completing the square shows that 
\[
\begin{aligned}
	&G^\frac{1}{p} (x-y,t-s) G(y-z,s)  \\
	&= \left( \frac{p}{p-1} \right)^\frac{n}{2} 
	\left( \frac{t}{t-s} \right)^\frac{n}{2p} 
	G^\frac{1}{p} (x-z,t) G\left( y-z, \frac{p}{p-1} s \right) \\
	&\quad \times 
	\exp\left( -\frac{t}{4ps(t-s)} \left| y-\frac{s}{t}x -\frac{t-s}{t} z \right|^2 
	\right) \\
	&\leq 
	C \left( \frac{t}{t-s} \right)^\frac{n}{2p} 
	G^\frac{1}{p} (x-z,t) G\left( y-z, \frac{p}{p-1} s \right), 
\end{aligned}
\]
and so by the change of variables $\tau=ps/(p-1)$, 
\[
\begin{aligned}
	&\int_0^{t/2} \int_{\R^n} G(x-y,t-s) U^p (y,s) dyds \\
	&\leq 
	C \int_0^{t/2} \int_{\R^n} 
	\left( \int_{\R^n}
	\left( \frac{t}{t-s} \right)^\frac{n}{2p} 
	G^\frac{1}{p} (x-z,t) G\left( y-z, \frac{p}{p-1} s \right)
	u_0(z) dz
	\right)^p 
	dyds \\
	&\leq 
	C \int_0^{t/2} \int_{\R^n} 
	\left( \int_{\R^n}
	G^\frac{1}{p} (x-z,t) G\left( y-z, \frac{p}{p-1} s \right)
	u_0(z) dz
	\right)^p 
	dyds \\
	&\leq 
	C \int_0^{pt/(2(p-1))} \int_{\R^n} 
	\left( \int_{\R^n}
	G^\frac{1}{p} (x-z,t) G\left( y-z, \tau \right)
	u_0(z) dz
	\right)^p 
	dyd\tau. 
\end{aligned}
\]
By expanding the time interval and 
replacing $\tau$ with $s$, the lemma follows. 
\end{proof}

Let us estimate the integral of $V$ which is given in Lemma \ref{lem:former}. 

\begin{lemma}\label{lem:Vp}
For $x\in \R^n$ and $t>0$, 
\[
	\int_0^\infty \int_{\R^n} 
	\left( V(x,y,t,s)
	\right)^p  dyds 
	\leq C \| u_0\|_X^{p-1} U(x,t). 
\]
\end{lemma}

\begin{proof}
For $m\geq2$, we estimate $\int_{I_m} \int_{\R^n} V^p(x,y,t,s) dyds$, 
where $I_m:=(1/m,m)$. 
By $V^p=V\times V^{p-1}$, \eqref{eq:Fdef}, Fubini's theorem 
and the H\"older inequality, 
\begin{equation}\label{eq:Vpiint}
\begin{aligned}
	&\int_{I_m} \int_{\R^n} V^p(x,y,t,s)  dyds \\
	&= 
	\int_{I_m} \int_{\R^n} 
	\int_{\R^n}
	G^\frac{1}{p} (x-z,t) G\left( y-z, s \right)
	u_0(z) dz
	V^{p-1}(x,y,t,s) dyds \\
	&= 
	\int_{\R^n}G^\frac{1}{p} (x-z,t) u_0(z) 
	\left[ 
	\int_{I_m} \int_{\R^n} 
	G\left( y-z, s \right)
	V^{p-1}(x,y,t,s) dyds \right] dz \\
	&\leq 
	\left( \int_{\R^n}G (x-z,t) u_0(z) dz\right)^\frac{1}{p} 
	\left( W(x,t) \right)^\frac{p-1}{p}, 
\end{aligned}
\end{equation}
where 
\[
	W(x,t):= \int_{\R^n} u_0(z) 
	\left[ 
	\int_{I_m} \int_{\R^n} 
	G\left( y-z, s \right)
	V^{p-1}(x,y,t,s) dyds \right]^\frac{p}{p-1} dz. 
\]

Let us estimate $W$. 
We take $\chi\in C_0^\infty(\R^n)$ 
such that $0\leq \chi\leq 1$, 
$\chi(x)=1$ for $|x|\leq 1$ and 
$\chi(x)=0$ for $|x|\geq 2$. 
For $R>1$, set $\chi_R(x):= \chi(x/R)$. 
By the definition of $X$-norm  
in \eqref{eq:Xdef} with a compactly supported smooth test function 
\[
	z\mapsto 
	\chi_R(z) 
	\left[ \int_{I_m} \int_{\R^n} 
	G\left( y-z, s \right)
	V^{p-1}(x,y,t,s) dyds \right], 
\]
and then by letting $R\to\infty$ 
with the help of the Gaussian decay, we see that 
\[
\begin{aligned}
	W(x,t) &\leq 
	\|u_0\|_X 
	\int_0^\infty \int_{\R^n} 
	\left| 
	\int_{\R^n} \Delta_\eta G(\eta-z,\tau) \right. \\
	&\quad \times 
	\left. \left[ \int_{I_m} \int_{\R^n} 
	G\left( y-z, s \right)
	V^{p-1}(x,y,t,s) dyds
	\right] dz \right|^\frac{p}{p-1} d\eta d\tau, 
\end{aligned}
\]
where we have changed the variables of integration 
from $(y,s)$ to $(\eta,\tau)$ 
in \eqref{eq:Xdef}. 
Then Fubini's theorem and the semigroup property yield 
\[
\begin{aligned}
	&\int_{\R^n} \Delta_\eta G(\eta-z,\tau) 
	\left[ \int_{I_m} \int_{\R^n} 
	G\left( y-z, s \right)
	V^{p-1}(x,y,t,s) dyds
	\right] dz \\
	&= 
	\Delta_\eta  \int_{I_m} \int_{\R^n} V^{p-1}(x,y,t,s) 
	\int_{\R^n} G(\eta-z,\tau) 
	G\left( z-y, s \right) dz dy ds \\
	&= 
	\Delta_\eta  \int_{I_m} \int_{\R^n} V^{p-1}(x,y,t,s) 
	G(\eta-y,s+\tau) dy ds \\
	&= 
	\int_{I_m} \int_{\R^n} \Delta_\eta G(\eta-y,s+\tau)
	V^{p-1}(x,y,t,s)  dy ds. 
\end{aligned}
\]

By the above estimates, 
\[
\begin{aligned}
	\frac{W(x,t)}{\|u_0\|_X }
	&\leq 
	\int_0^\infty \int_{\R^n} 
	\left| 
	\int_{I_m} \int_{\R^n} \Delta_\eta G(\eta-y,s+\tau) 
	V^{p-1}(x,y,t,s)  dy ds
	\right|^\frac{p}{p-1} d\eta d\tau \\
	&= 
	\int_0^\infty \left[ \int_{\R^n} 
	\left| 
	\int_{I_m} \int_{\R^n} \Delta_\eta G(\eta-y,s+\tau) 
	V^{p-1}  dy ds
	\right|^\frac{p}{p-1} d\eta \right]^{\frac{p-1}{p}\cdot \frac{p}{p-1}} d\tau. 
\end{aligned}
\]
Then, applying Minkowski's inequality in the right-hand side gives 
\[
\begin{aligned}
	&\frac{W}{\|u_0\|_X }
	\leq  
	\int_0^\infty \left[ 
	\int_{I_m} \left( \int 
	\left| \int  \Delta_\eta G(\eta-y,s+\tau) 
	V^{p-1}  dy \right|^\frac{p}{p-1} 
	d\eta \right)^\frac{p-1}{p}
	ds \right]^{\frac{p}{p-1}} d\tau, 
\end{aligned}
\]
where $\int(\cdots):=\int_{\R^n}(\cdots)$ and $V^{p-1}=V^{p-1}(x,y,t,s)$. 
From the H\"older inequality and a direct computation as 
\[
	\int_{\R^n} |\Delta_\eta G(\eta-y,s+\tau)|  dy 
	\leq C (s+\tau)^{-1}
\]
for a constant $C>0$ depending only on $n$, 
and then from Fubini's theorem 
with the same direct computation as above,  it follows that 
\[
\begin{aligned}
	&\int_{\R^n} 
	\left| \int_{\R^n}  \Delta_\eta G(\eta-y,s+\tau) 
	V^{p-1}  dy \right|^\frac{p}{p-1} 
	d\eta  \\
	&\leq 
	\int_{\R^n} 
	\left| \left( \int_{\R^n} |\Delta_\eta G(\eta-y,s+\tau)|  dy \right)^\frac{1}{p} 
	\left( \int_{\R^n}  |\Delta_\eta G(\eta-y,s+\tau)|
	V^p  dy \right)^\frac{p-1}{p} 
	\right|^\frac{p}{p-1} 
	d\eta \\
	&\leq 
	C(s+\tau)^{-\frac{1}{p-1}}
	\int_{\R^n} 
	\int_{\R^n}  |\Delta_\eta G(\eta-y,s+\tau)|
	V^p(x,y,t,s)  dy d\eta  \\
	&\leq 
	C(s+\tau)^{-\frac{p}{p-1}}
	\int_{\R^n} V^p(x,y,t,s)  dy, 
\end{aligned}
\]
where $C>0$ depends only on $n$ and $p$. 
Thus, 
\[
	\frac{W(x,t)}{\|u_0\|_X }
	\leq  
	C \int_0^\infty \left[ 
	\int_{I_m} \left( 
	(s+\tau)^{-\frac{p}{p-1}}
	\int_{\R^n} V^p(x,y,t,s)  dy
	\right)^\frac{p-1}{p}
	ds \right]^{\frac{p}{p-1}} d\tau. 
\]

Recall that $I_m:=(1/m,m)$. 
Then by the change of variables $s=\tau\sigma$, 
\[
\begin{aligned}
	&\frac{W(x,t)}{\|u_0\|_X }
	\leq  
	C \int_0^\infty \left[ 
	\int_{1/(m\tau)}^{m/\tau}
	\left( 
	(\sigma+1)^{-\frac{p}{p-1}}
	\int_{\R^n} V^p  dy
	\right)^\frac{p-1}{p}
	d\sigma \right]^{\frac{p}{p-1}} d\tau \\
	&= 
	C \left( 
	\int_0^\infty \left[ 
	\int_0^\infty 
	(1+\sigma)^{-1} 
	\chi_{(\frac{1}{m\tau}, \frac{m}{\tau})}(\sigma)
	\left( 
	\int_{\R^n} V^p  dy
	\right)^\frac{p-1}{p}
	d\sigma \right]^{\frac{p}{p-1}} d\tau 
	\right)^{ \frac{p-1}{p} \cdot \frac{p}{p-1} }, 
\end{aligned}
\]
where $V=V(x,y,t,\tau\sigma)$ 
and $\chi_A$ 
is the characteristic function on a set $A$. 
Then,  by Minkowski's inequality 
with a simple computation, 
and then by the change of variables $\tau=\sigma^{-1} s$ for fixed $\sigma$, 
\[
\begin{aligned}
	\frac{W(x,t)}{\|u_0\|_X }
	&\leq   
	C \left( 
	\int_0^\infty \left( 
	\int_0^\infty 
	(1+\sigma)^{-\frac{p}{p-1}} 
	\chi_{(\frac{1}{m\tau}, \frac{m}{\tau})}^\frac{p}{p-1}(\sigma)
	\int_{\R^n} V^p  dy
	d\tau \right)^{\frac{p-1}{p}} d\sigma 
	\right)^\frac{p}{p-1}  \\
	&= 
	C \left( 
	\int_0^\infty 
	(1+\sigma)^{-1}
	\left( 
	\int_{1/(m\sigma)}^{m/\sigma}
	\int_{\R^n} V^p (x,y,t,\tau\sigma) dy
	d\tau \right)^{\frac{p-1}{p}} d\sigma 
	\right)^\frac{p}{p-1} \\
	&= 
	C \left( 
	\int_0^\infty 
	(1+\sigma)^{-1}
	\left( 
	\int_{1/m}^m
	\int_{\R^n} V^p (x,y,t,s) dy 
	\frac{ds}{\sigma} \right)^{\frac{p-1}{p}} d\sigma 
	\right)^\frac{p}{p-1}, 
\end{aligned}
\]
and so by computing the integral with respect to $\sigma$, 
\[
\begin{aligned}
	\frac{W(x,t)}{\|u_0\|_X }
	&\leq   
	C \left( 
	\int_0^\infty 
	(1+\sigma)^{-1}
	\sigma^{-\frac{p-1}{p}} d\sigma 
	\right)^\frac{p}{p-1} 
	\int_{1/m}^m
	\int_{\R^n} V^p (x,y,t,s) dy ds \\
	&\leq 
	C \int_{1/m}^m
	\int_{\R^n} V^p (x,y,t,s) dy ds, 
\end{aligned}
\]
where $C>0$ depends only on $n$ and $p$. 

The above estimate for $W$ together with \eqref{eq:Vpiint} implies that 
\[
\begin{aligned}
	&\int_{1/m}^m \int_{\R^n} V^p(x,y,t,s)  dyds \\
	&\leq 
	C \|u_0\|_X^\frac{p-1}{p}  
	\left( \int_{\R^n}G (x-z,t) u_0(z) dz\right)^\frac{1}{p} 
	\left( \int_{1/m}^m
	\int_{\R^n} V^p (x,y,t,s) dy ds
	 \right)^\frac{p-1}{p}. 
\end{aligned}
\]
This implies that 
\[
	\left( 
	\int_{1/m}^m \int_{\R^n} V^p(x,y,t,s)  dyds \right)^\frac{1}{p} 
	\leq 
	C \|u_0\|_X^\frac{p-1}{p}  
	U^\frac{1}{p}(x,t). 
\]
Hence letting $m\to\infty$ shows the lemma. 
\end{proof}

We recall that we split 
the time integral in \eqref{eq:super} 
into $\int_0^t = \int_0^{t/2} + \int_{t/2}^t$. 
Let us estimate the second one. 

\begin{lemma}\label{lem:latter}
For $x\in \R^n$ and $t>0$, 
\[
	\int_{t/2}^t  \int_{\R^n} G(x-y,t-s) 
	U^p (y,s) dyds 
	\leq 
	\left( \int_{t/2}^t \| U (\cdot,s)\|_{L^\infty(\R^n)}^{p-1} ds \right)
	U(x,t). 
\]
\end{lemma}

\begin{proof}
By $U^p = U^{p-1}\times U$, \eqref{eq:Udef}, 
Fubini's theorem and the semigroup property, 
\[
\begin{aligned}
	&\int_{t/2}^t  \int_{\R^n} G(x-y,t-s) 
	U^p (y,s) dyds  \\
	&\leq 
	\int_{t/2}^t \| U (\cdot,s)\|_{L^\infty(\R^n)}^{p-1} 
	\int_{\R^n} G(x-y,t-s) 
	\int_{\R^n} G(y-z,s) u_0(z) dz dyds \\
	&=
	\int_{t/2}^t \| U (\cdot,s)\|_{L^\infty(\R^n)}^{p-1} 
	\int_{\R^n} u_0(z) \int_{\R^n}
	G(x-y,t-s) 
	G(y-z,s)  dy dz ds \\
	&= 
	\int_{t/2}^t \| U (\cdot,s)\|_{L^\infty(\R^n)}^{p-1} 
	\int_{\R^n} 
	G(x-z,t) u_0(z) dz ds. 
\end{aligned}
\]
Again by \eqref{eq:Udef}, the lemma follows. 
\end{proof}

In view of Lemma \ref{lem:latter}, 
we need a pointwise estimate of $U$.

\begin{lemma}\label{lem:Uesti}
For $x\in\R^n$ and $t>0$, 
\[
	U(x,t) \leq C \|u_0\|_X t^{-\frac{1}{p-1}}. 
\]
\end{lemma}

\begin{proof}
We note from \eqref{eq:Udef} that 
\[
	U(x,t)= \int_{\R^n} u_0(z) \left[ G^\frac{p-1}{p}(x-z,t) 
	\right]^\frac{p}{p-1} dz. 
\]
For $R>1$, 
we take $\chi_R\in C_0^\infty(\R^n)$ as in Lemma \ref{lem:Vp}. 
By the definition of $X$-norm  
in \eqref{eq:Xdef} with a compactly supported smooth test function 
\[
	z\mapsto 
	\chi_R(z) 
	G^\frac{p-1}{p}(x-z,t), 
\]
and then by letting $R\to\infty$ 
with the help of the Gaussian decay, 
\[
\begin{aligned}
	U(x,t) 
	&\leq 
	\|u_0\|_X \int_0^\infty \int_{\R^n} 
	\left| 
	\int_{\R^n} \Delta_y G(y-z,s) 
	G^\frac{p-1}{p} (x-z,t) dz \right|^\frac{p}{p-1} dy ds \\
	&=
	\|u_0\|_X \int_0^\infty \int_{\R^n} 
	\left| 
	\Delta_y  \int_{\R^n} G(y-z,s) 
	G^\frac{p-1}{p} (x-z,t) dz \right|^\frac{p}{p-1} dy ds. 
\end{aligned}
\]

From the definition of $G$ and the semigroup property, it follows that 
\[
\begin{aligned}
	&\int_{\R^n} G(y-z,s) 
	G^\frac{p-1}{p} (x-z,t) dz 
	= \int_{\R^n} G(y-z,s) 
	(4\pi t)^{-\frac{n(p-1)}{2p}}
	e^{-\frac{(p-1)|x-z|^2}{4pt}} dz  \\
	&= 
	\left( \frac{p}{p-1} \right)^{\frac{n}{2}} 
	(4\pi t)^{\frac{n}{2}-\frac{n(p-1)}{2p}}
	\int_{\R^n} G(y-z,s) 
	\left( 4\pi \left(\frac{p}{p-1} t \right) \right)^{-\frac{n}{2}}
	e^{-\frac{(p-1)|x-z|^2}{4pt}} dz  \\
	&= 
	\left( \frac{p}{p-1} \right)^{\frac{n}{2}} 
	(4\pi t)^\frac{n}{2p}
	\int_{\R^n} G(y-z,s) 
	G\left( z-x, \frac{p}{p-1} t\right) dz \\
	&= \left( \frac{p}{p-1} \right)^{\frac{n}{2}} 
	(4\pi t)^\frac{n}{2p}
	G\left( y-x, s+ \frac{p}{p-1} t\right). 
\end{aligned}
\]
Thus, by direct computation, 
\[
\begin{aligned}
	U(x,t) 
	&\leq 
	C \|u_0\|_X \int_0^\infty \int_{\R^n} 
	\left| 
	t^\frac{n}{2p}
	\Delta_y 
	G\left( y-x, s+ \frac{p}{p-1} t\right) \right|^\frac{p}{p-1} dy ds \\
	&= 
	C \|u_0\|_X t^\frac{n}{2(p-1)}
	\int_0^\infty \int_{\R^n} 
	\left| 
	\Delta_y 
	G\left( y, s+ \frac{p}{p-1} t\right) \right|^\frac{p}{p-1} dy ds \\
	&\leq 
	C \|u_0\|_X t^\frac{n}{2(p-1)}
	\int_0^\infty \int_{\R^n} 
	\left| 
	\left( s+ \frac{p}{p-1} t\right)^{-\frac{n}{2}-1}
	e^{-\frac{|y|^2}{8(s+ \frac{p}{p-1} t)}}
	\right|^\frac{p}{p-1} dy ds \\
	&\leq 
	C \|u_0\|_X t^\frac{n}{2(p-1)}
	\int_0^\infty 
	\left( s+ \frac{p}{p-1} t
	\right)^{\frac{n}{2}-(\frac{n}{2}+1)\frac{p}{p-1}} ds, 
\end{aligned}
\]
where $C>0$ depends only on $n$ and $p$. 
Since 
\[
\begin{aligned}
	\int_0^\infty 
	\left( s+ \frac{p}{p-1} t
	\right)^{\frac{n}{2}-(\frac{n}{2}+1)\frac{p}{p-1}} ds
	&= 
	t^{-\frac{n}{2(p-1)}-\frac{1}{p-1}} 
	\int_0^\infty 
	\left( \sigma+ \frac{p}{p-1} 
	\right)^{-\frac{n}{2(p-1)}-\frac{p}{p-1}} d\sigma \\
	&\leq 
	C t^{-\frac{n}{2(p-1)}-\frac{1}{p-1}}, 
\end{aligned}
\]
the lemma follows. 
\end{proof}

By combining the above estimates, we obtain the following: 

\begin{lemma}\label{lem:Duhamel}
For $x\in \R^n$ and $t>0$, 
\[
	\int_0^t \int_{\R^n} G(x-y,t-s) U^p(y,s) dyds 
	\leq C\|u_0\|_X^{p-1} U(x,t). 
\]
\end{lemma}

\begin{proof}
From Lemmas \ref{lem:former}, 
\ref{lem:Vp}, \ref{lem:latter} and \ref{lem:Uesti}, 
it follows that 
\[
\begin{aligned}
	&\int_0^t \int_{\R^n} G(x-y,t-s) U^p(y,s) dyds 
	= \int_0^{t/2}  \int_{\R^n} (\cdots)  
	+\int_{t/2}^t  \int_{\R^n} (\cdots)    \\
	&\leq 
	C \| u_0\|_X^{p-1} U(x,t)
	+ C \|u_0\|_X^{p-1} 
	\left( \int_{t/2}^t  s^{-1} ds \right)
	U(x,t). 
\end{aligned}
\]
This shows the lemma. 
\end{proof}

We check that $\overline{u}$ defined by \eqref{eq:super} 
is a supersolution. 

\begin{lemma}\label{lem:ou}
If $\|u_0\|_X$ is small depending only on $n$ and $p$, then  
$\overline{u}$ satisfies \eqref{eq:inteq}. 
\end{lemma}

\begin{proof}
From the definition of $\overline{u}$ in \eqref{eq:super} 
and Lemma \ref{lem:Duhamel}, it follows that 
\begin{equation}\label{eq:ouU}
	\overline{u}(x,t)
	\leq 
	(1+C\|u_0\|_X^{p-1}) U(x,t) 
\end{equation}
for $x\in\R^n$ and $t>0$. 
By applying this estimate to the definition of $\Phi$ in \eqref{eq:inteq}, 
\[
	\Phi[\overline{u}](x,t) 
	\leq  U(x,t) + (1+C'\|u_0\|_X^{p-1})^p \int_0^t \int_{\R^n} 
	G(x-y,t-s)  U^p(y,s)  dyds, 
\]
where $C'>0$ depends only on $n$ and $p$. 
Therefore, when $\|u_0\|_X$ is small, 
we see again from the definition of $\overline{u}$ in \eqref{eq:super} that 
\[
	\Phi[\overline{u}](x,t) 
	\leq  U(x,t) + 2^p \int_0^t \int_{\R^n} 
	G(x-y,t-s) U^p(y,s)  dyds 
	= \overline{u}(x,t). 
\]
Hence $\overline{u}$ is a supersolution. 
\end{proof}

We are now in a position to prove Proposition \ref{pro:ex}. 

\begin{proof}[Proof of Proposition \ref{pro:ex}]
Once we construct a supersolution, we can also construct a solution, 
see \cite[Section 2]{RS13} for instance.
Thus, we give an outline. 
By Lemma \ref{lem:ou} and \eqref{eq:ouU}, 
there exists $c>0$ such that 
if $\|u_0\|_X\leq c$, then $\overline{u}$ satisfies \eqref{eq:inteq} and 
$\overline{u}\leq 2U$. 
Set $u_1(x,t):= U(x,t)$ and 
\[
	u_k(x,t) := 
	U(x,t) + \int_0^t \int_{\R^n} G(x-y,t-s) u_{k-1}^p (y,s) dyds 
	\quad (k=2,3, \ldots). 
\]
Then, we can check inductively that 
\[
	0\leq u_1(x,t) \leq u_2(x,t) \leq 
	\cdots \leq u_k(x,t) \leq \cdots \leq \overline{u}(x,t) \leq 2U(x,t) 
\]
for $x\in\R^n$ and $t>0$. 
Then, the pointwise limit $u_*(x,t):=\lim_{k\to\infty} u_k(x,t)$ 
exists for each $x\in \R^n$ and $t>0$. 
The monotone convergence theorem guarantees that 
$u_*=\Phi[u_*]$ in $\R^n\times (0,\infty)$. 
By  $u_*\leq \overline{u}\leq 2U$ and Lemma \ref{lem:Uesti}, 
\[
	u_*(x,t)\leq C \|u_0\|_X t^{-\frac{1}{p-1}} 
	\quad \mbox{ for } t>0. 
\]
Moreover, since $0\leq u_* \leq 2U \leq 2\|u_0\|_{L^\infty(\R^n)}$, 
we can check that 
$u_*$ is a classical solution, 
and so uniqueness is standard. 
The proof is complete. 
\end{proof}

Finally, we prove our main result.

\begin{proof}[Proof of Theorem \ref{th:main}]
Proposition \ref{pro:esti} implies that 
if $\|u_0\|_X\geq C_p$ for some $C_p>0$ depending only on $p$, 
then  $T_*(u_0)<\infty$. 
Moreover, Proposition \ref{pro:ex} with uniqueness 
also implies that 
if $\|u_0\|_X\leq c_{n,p}$ for some $c_{n,p}>0$ 
depending only on $n$ and $p$, 
then  $T_*(u_0)=\infty$. 
Hence $\|\cdot\|_X$ is a separating norm. 
By the $1$-homogeneity of norms, 
$\|\cdot\|_X$ is unique up to equivalence on nonnegative data. 
The proof is complete. 
\end{proof}

\subsection*{Acknowledgments}
The author was supported 
by JSPS KAKENHI Grant Numbers 22H01131, 22KK0035 and 23K12998.

\subsection*{AI Usage Disclosure}
ChatGPT was used in the preparation and revision of this manuscript. 
The author reviewed the entire manuscript, made substantial revisions throughout 
the manuscript and takes full responsibility for the final manuscript.

\end{document}